\documentclass[a4paper,12pt]{article}
\usepackage[a4paper]{geometry}
\usepackage{amsmath}
\usepackage{amsthm}
\usepackage{amssymb}
\usepackage{lscape}
\usepackage{rotating}
\usepackage{bm}
\usepackage{graphicx}
\newcommand{\GL}{\ensuremath \mathrm{GL}}
\newcommand{\GO}{\ensuremath \mathrm{GO}}
\newcommand{\SO}{\ensuremath \mathrm{SO}}
\newcommand{\Sp}{\ensuremath \mathrm{Sp}}

\newcommand{\SL}{\ensuremath \mathrm{SL}}

\newcommand{\St}{\ensuremath \mathrm{St}}
\newcommand{\tr}{\ensuremath \mathrm{tr}}
\newcommand{\F}{\ensuremath \mathbb{F}}

\newcommand{\Irr}{\ensuremath \mathrm{Irr}}

\makeatletter
\newcommand{\maop}[1]{%
\ensuremath{\mathop{\operator@font #1}\nolimits}}
\newcommand{\maopl}[1]{%
\ensuremath{\mathop{\operator@font #1}\limits}}
\makeatother

\newcommand{\infl}{\maop{Infl}}

\newcommand{\R}[2]{\maop{R}^{#1}_{#2}}

\newcommand{\downa}[2]{\smash{\maop{\downarrow}^{#1}_{#2}}}
\newcommand{\upa}[2]{\smash{\maop{\uparrow}^{#1}_{#2}}}
\newcommand{\mbv}{\ensuremath \mathbf{v}}
\newcommand{\mbw}{\ensuremath \mathbf{w}}
\newcommand{\mbu}{\ensuremath \mathbf{u}}
\newcommand{\mbs}{\ensuremath \mathbf{s}}
\newcommand{\mbx}{\ensuremath \mathbf{x}}
\newcommand{\mbz}{\ensuremath \mathbf{z}}

\theoremstyle{plain}
\newtheorem{theorem}{Theorem}[section]
\newtheorem{proposition}[theorem]{Proposition}
\newtheorem{lemma}[theorem]{Lemma}
\newtheorem{corollary}[theorem]{Corollary}

\theoremstyle{definition}

\theoremstyle{remark}
\newtheorem{remark}[theorem]{Remark}

\date{\today}
\author{Frank Himstedt and Felix Noeske}
\title{Restricting unipotent characters in special orthogonal groups}
\begin{document}
\maketitle
\begin{abstract}
For all prime powers $q$ we restrict the unipotent characters of the
special orthogonal groups $\SO_5(q)$ and $\SO_7(q)$ to a maximal
parabolic subgroup. We determine all irreducible constituents of these
restrictions for $\SO_5(q)$ and a large part of the irreducible
constituents for $\SO_7(q)$. 
\end{abstract}

\section{Introduction}

Among the ordinary irreducible characters of a finite group $G$ of Lie
type the unipotent characters possess some remarkable
properties. For example, the Jordan decomposition of characters
gives a connection between the ordinary irreducible characters of $G$
and the unipotent characters of certain subgroups of the dual group.
Furthermore, if $\ell$ is a prime different from the defining
characteristic and not too small then the reductions modulo $\ell$ of
the unipotent irreducible characters form a so-called basic set for
the unipotent $\ell$-blocks of $G$, so that knowledge of the 
decomposition numbers of the unipotent irreducible characters can be
used to derive all decomposition numbers of the unipotent blocks. 

The analysis of the restriction of representations to maximal
parabolic subgroups is an important tool in the representation theory
of finite groups of Lie type. One reason for this is that maximal
parabolic subgroups are large subgroups and another one is that the
Levi decomposition in conjunction with Clifford theory often allows
for a reduction of representation theoretical questions to groups of
smaller rank. 

Let $q$ be a prime power. In this paper we study the restriction of
the unipotent irreducible characters of the special orthogonal groups 
$G = \SO_5(q)$ and $G=\SO_7(q)$ to the maximal parabolic subgroup $P$
which is defined as the stabilizer in $G$ of a one-dimensional
subspace of the natural module. The irreducible characters of the
parabolic subgroup $P$ are partitioned into Types $1$, $0$, $+$, $-$ via
Clifford theory. For $G=\SO_5(q)$ we determine all irreducible
constituents of the restrictions of the unipotent irreducible
characters of $G$ to $P$. For $G=\SO_7(q)$ we obtain complete
information on the irreducible constituents of Types $1$, $0$ and
partial information on the constituents of Types $\pm$. Our motivation
lies in the computation of the decomposition numbers of $\SO_7(q)$ in
non-defining characteristic. The results we obtain in this paper
contribute to solving this task in \cite{HimNoeDecnosSp6So7}.
For even $q$ the special orthogonal groups are
isomorphic to symplectic groups and in this case 
restrictions of representations in non-defining characteristic to
maximal subgroups were previously investigated in~\cite{GuralnickTiep} and
\cite{Schaeffer}, for example.

Based on our motivation to compute and compare the decomposition
numbers of $\Sp_6(q)$ and $\SO_7(q)$, this paper and the paper 
\cite{AnHissUnipotent} are kindred spirits. In fact, large parts of
our paper are guided by the work of An and Hiss. However, there are
some remarkable differences: For the orthogonal groups the unipotent
radical of $P$ is abelian, while it is non-abelian for the parabolic
subgroup considered in~\cite{AnHissUnipotent}. Furthermore, there are
structural differences between the inertia subgroups in the parabolic
subgroup of the orthogonal and the symplectic groups, respectively. 
And finally, we also consider even prime powers $q$. 

The paper is structured as follows: We fix notation for the orthogonal
groups in Section~\ref{sec:SO} and introduce the maximal parabolic
subgroup $P$. In Section~\ref{sec:irrP} we describe the construction
and parameterization of the irreducible characters of~$P$ via Clifford
theory. Section~\ref{sec:preind} is the technical heart of this
paper. In Theorem~\ref{thm:indnu} we obtain a description of the
restriction of those characters of $G$ which are Harish-Chandra
induced from the standard Levi subgroup of $P$. In most cases this
reduces the decomposition of the Harish-Chandra induced character to a
similar problem for smaller subgroups. In Section~\ref{sec:restrictionSt} 
we collect some general results on the restriction of the Steinberg 
character. Section~\ref{sec:charvalU} deals with the values of the
unipotent irreducible characters of the special orthogonal groups on
certain unipotent conjugacy classes. We show that these values already
determine the degrees of the components of Types~$1$, $0$, $+$ and $-$
of the restrictions of the unipotent characters to~$P$.
The main results of this paper are contained in
Sections~\ref{sec:resuniSO5} and \ref{sec:resuniSO7} where we consider
the restrictions of the unipotent irreducible characters of
$G=\SO_5(q)$ and $G=\SO_7(q)$ to the maximal parabolic subgroup
$P$. Since $\SO_7(q)$ does not have any cuspidal unipotent character
the crucial tool is the description of the restriction of Harish-Chandra
induced modules in Theorem~\ref{thm:indnu}.

\section{Special orthogonal groups}
\label{sec:SO}

In this section we collect some information on the special orthogonal
groups $\SO_n(q)$ for odd $n$ and on a maximal parabolic subgroup
$P_n$ of $\SO_n(q)$. For more information on orthogonal groups we
refer to~\cite[Chapter~11]{Taylor}. 

\subsection{Special orthogonal groups}
\label{subsec:SO}

Let $q$ be a power of a prime $p$ and $\F_q$ a finite field 
with $|\F_q|=q$. We fix a positive integer $m$. Let
$I_m \in \F_q^{m \times m}$ be the identity matrix and
$J_m \in \F_q^{m \times m}$ the matrix with ones on the 
anti-diagonal and zeros elsewhere.
We write $\mbv \in \F_q^{2m+1}$, $\mbw \in \F_q^{2m}$ as
\[
\mbv =\begin{bmatrix}v_m, \dots, v_1, v_0, v_1', \dots,
v_m'\end{bmatrix}^\tr , \quad 
\mbw =\begin{bmatrix}w_m, \dots, w_1, w_1', \dots,
w_m'\end{bmatrix}^\tr.
\]
Fix an element $\nu \in \F_q$ such that the polynomial 
$X^2+X+\nu \in \F_q[X]$ is irreducible. We define quadratic forms
$Q_{2m+1}$ on $\F_q^{2m+1}$ and $Q_{2m}^\pm$ on $\F_q^{2m}$ by
\begin{eqnarray*}
Q_{2m+1}(\mbv) & := & v_0^2 + \sum_{j=1}^m v_j v_j', \quad
Q_{2m}^+(\mbw) := \sum_{j=1}^m w_j w_j' \quad \text{and}\\
Q_{2m}^-(\mbw) & := & w_1^2 + w_1 w_1' + \nu w_1'^2 + \sum_{j=2}^m w_j w_j'.
\end{eqnarray*}
Let $\{e_m, \dots, e_1,e_0,e_1', \dots, e_m'\}$ be the standard basis
of the vector space $\F_q^{2m+1}$ and let $\{f_m, \dots, f_1, f_1', \dots, 
f_m'\}$ be the standard basis of $\F_q^{2m}$. We set
\[
J_{2m+1}' := \begin{bmatrix}. & . & J_m\\. & 2 & .\\J_m & . &
  .\end{bmatrix}, \quad J_{2m}^+ := J_{2m}, \quad 
J_{2m}^- := \begin{bmatrix}. & . & . & J_{m-1}\\. & 2 & 1 & .\\. & 1
 & 2\nu & .\\J_{m-1} & . & . & .\end{bmatrix}.
\]
Throughout the paper we usually write dots for zeros as matrix entries, or
omit them to increase legibility.
For all $\mbv, \mbv' \in \F_q^{2m+1}$ and $\mbw, \mbw' \in \F_q^{2m}$
we have 
\begin{eqnarray*}
\mbv^\tr J_{2m+1}' \mbv'  & = & Q_{2m+1}(\mbv+\mbv') - Q_{2m+1}(\mbv) - Q_{2m+1}(\mbv') ,\\
\mbw^\tr J_{2m}^\pm \mbw' & = & Q^\pm_{2m}(\mbw+\mbw') - Q^\pm_{2m}(\mbw) - Q^\pm_{2m}(\mbw').
\end{eqnarray*}
Hence, $J_{2m+1}'$ is the Gram matrix of the polar form of $Q_{2m+1}$
with respect to the standard basis of $\F_q^{2m+1}$, and $J^\pm_{2m}$
is the Gram matrix of the polar form of $Q^\pm_{2m}$ with respect to
the standard basis of $\F_q^{2m}$; see~\cite[Chapter~11]{Taylor}. 
We define
\begin{eqnarray*}
\GO_{2m+1}(q) & := & \{ \mbx \in \GL_{2m+1}(q) \mid Q_{2m+1}(\mbx\mbv)
= Q_{2m+1}(\mbv) \,\, \text{for all} \,\, \mbv \in \F_q^{2m+1} \},\\
\GO^\pm_{2m}(q) & := & \{ \mbx \in \GL_{2m}(q) \mid Q^\pm_{2m}(\mbx\mbw)
= Q^\pm_{2m}(\mbw) \,\, \text{for all} \,\, \mbw \in \F_q^{2m} \},\\
\SO_{2m+1}(q) & := & \GO_{2m+1}(q) \cap \SL_{2m+1}(q) , \quad
\SO^\pm_{2m}(q) := \GO^\pm_{2m}(q) \cap \SL_{2m}(q),
\end{eqnarray*}
and use the convention $\GO^\pm_0(q) := \SO^\pm_0(q) := \{1\}$. 
The orders of these groups are:
\begin{eqnarray*}
|\GO_{2m+1}(q)|   & = & d \cdot q^{m^2}(q^{2m}-1)(q^{2m-2}-1) \cdots (q^2-1),\\
|\SO_{2m+1}(q)|   & = & q^{m^2}(q^{2m}-1)(q^{2m-2}-1) \cdots (q^2-1),\\
|\GO^\pm_{2m}(q)| & = & 2 q^{m(m-1)} (q^m \mp 1)(q^{2m-2}-1)(q^{2m-4}-1) \cdots (q^2-1),\\
|\SO^\pm_{2m}(q)| & = & {\textstyle \frac2e} \cdot q^{m(m-1)} 
(q^m \mp 1)(q^{2m-2}-1)(q^{2m-4}-1) \cdots (q^2-1),
\end{eqnarray*}
where $d := \gcd(2, q-1)$ and $e := \gcd(2, q^m \mp 1)$. Note that for even $q$ we have
$\GO_{2m+1}(q) = \SO_{2m+1}(q)$ and $\GO^\pm_{2m}(q) = \SO^\pm_{2m}(q)$. 
This agrees with the definition in the ATLAS~\cite{Atlas}, but it
differs from the one in~\cite{Taylor}. 

\begin{remark} \label{rmk:even_iso_Sp}
For even $q$ there is a natural isomorphism 
\[
\SO_{2m+1}(q) \to \Sp_{2m}(q) := \{ \mbx \in \GL_{2m}(q) \mid \mbx^\tr
J_{2m} \mbx = J_{2m} \}
\]
mapping each matrix $A$ to the matrix which is obtained from $A$ by
removing the middle row and the middle column.
\end{remark}

\subsection{The Weyl group}
\label{subsec:weyl}

From now on we fix an odd integer $n=2m+1$ ($m\ge 1$) and set
$G:=G_n:=\SO_n(q)$. The Weyl group $W$ of $G$ is of type $B_m$ and 
we number the simple roots such that the Dynkin diagram is:
\setlength{\unitlength}{0.9mm}
\begin{center}
\begin{picture}(85,13)
\thinlines
\put(-10,5){\circle*{1.5}}
\put(15,5){\circle*{1.5}}
\put(40,5){\circle*{1.5}}
\put(75,5){\circle*{1.5}}
\put(100,5){\circle*{1.5}}
\put(-8,5.5){\line(1,0){21}}
\put(-8,4.5){\line(1,0){21}}
\put(17,5){\line(1,0){21}}
\put(42,5){\line(1,0){10}}
\put(63,5){\line(1,0){10}}
\put(55,3.85){$\cdots$}
\put(77,5){\line(1,0){21}}
\put(5,7){\line(-2,-1){4}}
\put(5,3){\line(-2,1){4}}
\put(-11.2,8){$1$}
\put(14,8){$2$}
\put(39,8){$3$}
\put(69,8){$m-1$}
\put(98,8){$m$}
\end{picture}
\end{center}
So the first simple root is short, the others are long. Let 
$s_j \in W$ be the reflection at the $j$-th simple root,
$j=1,2,\dots,m$. An inverse image (which we also denote by $s_j$) in
the normalizer~$N_G(T)$ of the maximally split torus 
\begin{equation} \label{eq:torus}
T := \{
\mathrm{diag}(t_m^{},\ldots,t_1^{},1,t_1^{-1},\ldots,t_m^{-1})
\mid t_1, \dots, t_m \in \F_q^\times\} 
\end{equation}
of $G$ is given by
\begin{eqnarray}
s_1 & := & \begin{bmatrix} I_{m-1} &  &  &  &  \\
 &.&.&1&\\&.&-1&.&\\&1&.&.&\\&&&&I_{m-1}\end{bmatrix}
\quad \text{and} \label{eq:s_1} \\
s_j & := & \begin{bmatrix}I_{m-j} &   &   &   &  \\  & J_2 &   &   &
   \\  &   & I_{2j-3} &   &  \\  &   &   & J_2 &  \\  &   &   &   &
  I_{m-j}\end{bmatrix} \quad \text{for} \ j \in \{2,3,\dots,m\}.
\label{eq:s_j}
\end{eqnarray}
For $m \ge 2$ we set $s := s_m$ and 
$t := s_m s_{m-1} \cdots s_2 s_1 s_2 \cdots s_{m-1} s_m$. It follows
from (\ref{eq:torus})-(\ref{eq:s_j}) that we can choose the following
inverse images of $s$ and $t$ in $N_G(T)$:
\begin{equation}\label{eq:defst}
s = \begin{bmatrix} J_2 &   &  \\  & I_{n-4} &  \\   &   &J_2\end{bmatrix}
\quad \text{and} \quad
t = \begin{bmatrix} &  &  &  & 1\\  & I_{m-1} & . & . & \\
 & . & -1 & . & \\  & . & . & I_{m-1} & \\ 
1 &  &  &  & \end{bmatrix}.
\end{equation}

\subsection{A parabolic subgroup}
\label{subsec:Pn}

Let $P_n$ denote the stabilizer in $G_n$ of the $1$-dimensional
subspace generated by the vector 
$e_m = \begin{bmatrix}1,0,0,\dots,0\end{bmatrix}^\tr$ in the natural
module for $G_n$. Then $P_n$ is a maximal parabolic 
subgroup of order $|P_n| = q^{m^2}(q-1)(q^{2m-2}-1)(q^{2m-4}-1) \cdots
(q^2-1)$ with Levi decomposition $P_n= U_n \rtimes L_n$, where 
\[
L_n  =  \{\mathbf{s}_n(\mathbf{x}, a) \mid \mathbf{x} \in G_{n-2},
 a\in\F_q^\times\}, \quad
U_n  = \{\mathbf{u}_n(\mathbf{v}) \mid \mathbf{v} \in \F_q^{n-2}\}
\]
and
\begin{equation}
 \mathbf{s}_n(\mathbf{x}, a) := \begin{bmatrix}a &   &  \\   &
   \mathbf{x} &  \\   &   & a^{-1} \end{bmatrix}, \quad 
\mathbf{u}_n(\mathbf{v}) := \begin{bmatrix} 1 & -\mathbf{v}^\tr
  J_{n-2}' & -Q_{n-2}(\mbv)\\   & 1 & \mathbf{v}\\   &   & 1 \end{bmatrix} 
\end{equation}
The Levi subgroup $L_n$ is a direct product
$L_n = L_n' \times A \cong \SO_{n-2}(q) \times \F_q^\times$ of 
\[
L_n' := \{\mathbf{s}_n(\mathbf{x}, 1) \mid \mathbf{x} \in \SO_{n-2}(q)\}
\ \text{and} \ A := \{\mathbf{s}_n(I_{n-2}, a) \mid a \in \F_q^\times \}.
\]
The unipotent radical $U_n$ of $P_n$ is an elementary abelian group of
order $q^{n-2}$. Often we consider $G_{n-2}$ as a subgroup of $G_n$
and $P_{n-2}$ as a subgroup of $L_n'$ via the identification 
$\mbx \mapsto \mbs_n(\mbx, 1)$.

\section{The irreducible characters of $P_n$}
\label{sec:irrP}

In this section we fix notation and classify the irreducible
characters of the maximal parabolic subgroup $P_n$ via Clifford
theory. A similar classification was also obtained by Schm\"olzer
in~\cite[Chapter~2]{Schmoelzer}.

\subsection{General character theoretic notation}

Let $K$ be a subgroup of a finite group $H$. We write $\Irr(H)$ for
the set of complex irreducible characters of $H$ and $1_H$ for the
trivial character. Let $\langle \cdot,\cdot \rangle_H$ be the usual
scalar product on the space of class functions of $H$. 
If $\chi$ is a character of $H$ we write 
$\chi\smash\downarrow^H_K$ for the restriction of $\chi$ 
to~$K$, and if~$\varphi$ is a character of $K$ we write
$\varphi\smash\uparrow_K^H$ for the character of $H$ which is 
induced by~$\varphi$. If~$K \unlhd H$ and $\psi$ is a
character of the factor group $H/K$ then we denote its inflation
to~$H$ by $\infl_{H/K}^H \psi$. For two characters $\chi, \psi$ of $H$ we
say that $\psi$ is a subcharacter of $\chi$ if $\chi-\psi$ is a character. 

\subsection{Inertia subgroups of $P_n$} 
\label{subsec:inertiasubgrps}

Let $n \ge 5$. The conjugation action of $L_n$ on $U_n$ is
given by ${}^{\mathbf{s}_n(\mathbf{x}, a)} \mathbf{u}_n(\mathbf{v}) =
\mathbf{u}_n(a\mathbf{xv})$. As before, we write 
$\mathbf{v} \in \F_q^{n-2}$ as $\mathbf{v} = \begin{bmatrix}v_{m-1},
\dots, v_1, v_0, v_1', \dots, v_{m-1}'\end{bmatrix}^\tr$, and fix a
non-trivial irreducible complex character $\xi$ of $(\F_q,+)$. The
corresponding action of $L_n$ on $\Irr(U_n)$ has four orbits. We
choose as representatives $1_U$ (the trivial character), and
non-trivial $\lambda^0$, $\lambda^+$, $\lambda^-$. For 
$\varepsilon \in \{0, +, -\}$ we denote the inertia subgroup of to
$\lambda^\varepsilon$ by $I^\varepsilon$. We define $\lambda^0$ by
$\lambda^0(\mathbf{u}_n(\mathbf{v})) := \xi(v_{m-1}')$ 
and get $I^0 = U_n\tilde{P}_{n-2}$ with $\tilde{P}_{n-2} :=
  \tilde{U}_{n-2} \tilde{L}_{n-2}$ where 
\begin{equation}
\tilde{L}_{n-2} := \left\{
\begin{bmatrix}a & & & & \\ & a &   &   &\\ &   &
  \mathbf{x} &   & \\ &   &   & a^{-1} &\\ & & & & a^{-1}\end{bmatrix} 
\mid \mathbf{x} \in \SO_{n-4}(q), a \in \F_q^\times \right\} 
\end{equation}
and 
\begin{equation}
\tilde{U}_{n-2} := \left\{\begin{bmatrix} 1 & & \\ 
  & \mathbf{u}_{n-2}(\mathbf{v}) & \\ & & 1
 \end{bmatrix} \mid \mathbf{v} \in \F_q^{n-4} \right\}. 
\end{equation}
Furthermore $|\tilde{P}_{n-2}| = q^{(m-1)^2}(q-1)(q^{2m-4}-1)(q^{2m-6}-1) \cdots (q^2-1)$.
We define~$\lambda^+$ by $\lambda^+(\mathbf{u}_n(\mathbf{v})) :=
\xi(v_0)$ and so $I^+ = U_n L_n^+$ where
\begin{equation}
L_n^+ := \left\{\begin{bmatrix}a & & & & \\ & A & . & B &\\ & . &
  a & . & \\ & C & . & D &\\ & & & & a\end{bmatrix} \mid
\begin{bmatrix} A & B\\ C & D\end{bmatrix}   \in \GO^+_{n-3}(q),
  a=\det\begin{bmatrix} A & B\\ C & D\end{bmatrix} \right\}.
\end{equation}
Note that $a = a^{-1} = \pm 1$. In particular, $L_n^+ \cong \GO_{n-3}^+(q)$ 
and 
\begin{equation}
|L_n^+|=2q^{(m-1)(m-2)}(q^{m-1}-1)(q^{2m-4}-1)(q^{2m-6}-1) \cdots (q^2-1).
\end{equation}
The construction of $\lambda^{-}$ is more complicated and less
explicit. We define a quadratic form $Q_n'$ on $\F_q^n$ by
\[
Q_n'(\mbv) := v_0^2 + v_1^2 + v_1 v_1' + \nu v_1'^2 +
\sum_{j=2}^m v_j v_j'. 
\]
It follows from~\cite[p.~139,~II.]{Taylor} that there is 
$\nu' \in \F_q^\times$ and 
\[
\mathbf{b}_n = \begin{bmatrix}\nu' I_{m-1} &   &  \\  & \mathbf{b}_3' &  \\
  &   & I_{m-1}\end{bmatrix} \in \GL_n(q)
\]
such that $Q_n'(\mathbf{b}_n\mbv) = \nu' \cdot Q_n(\mbv)$ for all
$\mbv \in \F_q^n$. Note that the matrix $\mathbf{b}_3' \in \GL_3(q)$
depends on $q$ but not on $n$. For odd $q$ we can choose 
$\nu'$ to be~$1$ or a  non-square in~$\F_q^\times$; for even $q$ we can
always choose $\nu'=1$. We define
$\lambda^{-}(\mathbf{u}_n(\mathbf{v})) := \xi((\mathbf{b}_{n-2}\mathbf{v})_0)$. 
A~straightforward calculation shows that $\mbs_n(\mathbf{y}, a) \in
I^{-}$ if and only if 
\[
\mathbf{y} \in \left\{\mathbf{b}_{n-2}^{-1}\begin{bmatrix} A & . & B \\ . &
  a & . \\ C & . & D \end{bmatrix}\mathbf{b}_{n-2} \mid
\begin{bmatrix} A & B\\ C & D\end{bmatrix} \in \GO^-_{n-3}(q),
  a=\det\begin{bmatrix} A & B\\ C & D\end{bmatrix} \right\}.
\]
Again $a = a^{-1} = \pm 1$. Thus $I^{-} = U_n L_n^-$ where 
$L_n^- \cong \GO_{n-3}^-(q)$. Furthermore
\begin{equation}
|L_n^-|=2q^{(m-1)(m-2)}(q^{m-1}+1)(q^{2m-4}-1)(q^{2m-6}-1) \cdots (q^2-1).
\end{equation}

Since the characters $\lambda^0$, $\lambda^+$, and 
$\lambda^-$ are non-trivial, the orders of their inertia subgroups 
are pairwise distinct, and the sizes of the four orbits
add up to $q^{n-2}$, the set  
$\{1_U, \lambda^0, \lambda^{+}, \lambda^{-}\}$ is a
set of representatives for the orbits of $L_n$ on $\Irr(U_n)$.

When $n=3$ the orbits of $\lambda^0$, $\lambda^-$ do not exist and
we have $L_3^+ := \GO_0^+(q) = \{1\}$ by definition.

\subsection{Irreducible characters of $P_n$}
\label{subsec:charsPn}

We still assume that $q$ is an arbitrary prime power and that 
$n=2m+1 \ge 3$. For simplicity we set $P := P_n$, $L := L_n$, 
$U := U_n$. We obtain four types of irreducible characters of $P$ 
according to which character of $1_U$, $\lambda^0$, and
$\lambda^{\pm}$ they cover:

\medskip

\begin{tabular}{ll}
Type $1$: & the characters with $U$ in their kernel,\\
Type $0$: & the characters covering $\lambda^0$,\\
Type $+$: & the characters covering $\lambda^{+}$, \\
Type $-$: & the characters covering $\lambda^{-}$.
\end{tabular}

\medskip

Note that for $n=3$ there are no characters of Type $0$ or Type~$-$. 
The characters of Type 1 are parameterized by the irreducible
characters of $L$ via inflation and we write 
${^1}\psi_\sigma := \infl_L^P \sigma$ for $\sigma \in \Irr(L)$. We
have ${^1}\psi_\sigma(1) = \sigma(1)$. 

Since $\lambda^0$, $\lambda^{\pm}$ are linear and 
$I^0 = U \rtimes \tilde{P}_{n-2}$, 
$I^{\pm} = U \rtimes L_n^\pm$ the characters $\lambda^0$,
$\lambda^{\pm}$ can be extended trivially to characters
$\hat{\lambda}^0 \in \Irr(I^0)$ and 
$\hat{\lambda}^{\pm} \in \Irr(I^{\pm})$, respectively.
We have a bijection between $\Irr(P_{n-2})$ and
$\Irr(\tilde{P}_{n-2})$ sending $\mu \in \Irr(P_{n-2})$ to the
character $\tilde{\mu} := (\mu \boxtimes 1_A)\smash\downarrow^{P_{n-2} 
  \times A}_{\tilde{P}_{n-2}}$. The irreducible characters of $P$ of
Type 0 are parameterized by $\Irr(P_{n-2})$ and the character of $P$
corresponding to $\mu \in \Irr(P_{n-2})$ is 
${^0}\psi_\mu:=( \hat{\lambda}^0 \cdot \infl_{\tilde{P}_{n-2}}^{I^0}
\tilde{\mu})\smash\uparrow_{I^0}^P$. We have 
${^0}\psi_\mu(1) = (q^{2m-2}-1) \mu(1)$. The irreducible characters of
$P$ of Type $\pm$ are parameterized by $\Irr(\GO_{n-3}^\pm(q))$ and the
character of $P$ corresponding to $\vartheta \in \Irr(\GO_{n-3}^\pm(q))$ 
is ${^\pm}\psi_\vartheta := (\hat{\lambda}^\pm \cdot
\infl_{L_n^\pm}^{I^\pm}\vartheta)\smash\uparrow_{I^\pm}^P$.
For the degrees of these characters, we have ${^\pm}\psi_\vartheta(1)
=  \frac12 q^{m-1} (q^{m-1}\pm1)(q-1) \vartheta(1)$. 

Let $\varepsilon \in \{1,0,+,-\}$. In the same way
as~\cite[Section~2]{AnHissUnipotent}, we use additive extension to 
expand the notation ${}^\varepsilon\psi_\sigma$ to non-irreducible
characters $\sigma$. For example, if $\sigma = \sum m_j \sigma_j$ 
with $\sigma_j \in \Irr(L)$ we set 
${}^1\psi_\sigma := \sum m_j {}^1\psi_{\sigma_j}$.

\section{Preliminary results on induced characters}
\label{sec:preind}

In this section we provide some information on certain characters of
$P$ which are induced from various subgroups. These results will be
used in subsequent sections when we study of the restriction of
Harish-Chandra induced characters via Mackey's Theorem.

\subsection{Group theoretical lemmas}

We use the setting and notation from Sections~\ref{sec:SO} and
\ref{sec:irrP}. In particular, we fix an odd integer $n=2m+1$ where 
$m \ge 1$. We choose simple reflections $s_1, s_2, \dots, s_m$
generating the Weyl group $W$ of $G_n$ and define elements 
$s, t \in W$ as in Section~\ref{subsec:weyl}. 
The following results are similar to~\cite[Lemmas~3.1, 3.2]{AnHissUnipotent}.  

\begin{lemma} \label{la:intersec_par}
Assume that $m \ge 2$. Let $J := \{s_1, s_2, \dots, s_{m-1}\}$ and 
$s, t \in G_n$ as in Section~\ref{subsec:weyl}. We write $W_J$ for the
subgroup of $W$ generated by $J$ and $D_{J,J}$ for the set of
distinguished representatives for $W_J \backslash W / W_J$ as in
\cite[Section~2.7]{Carter} and define $K := {}^sJ \cap J$. Then:
\begin{enumerate}
\item[(a)] $D_{J,J} = \{1,s,t\}$.

\item[(b)] ${}^tP \cap P = L$.

\item[(c)] $K = \{s_1, \dots, s_{m-2}\}$ and
${}^sP \cap P = ({}^sU \cap U)({}^sL \cap U)({}^sU \cap L)L_K$.
Moreover, $Q_K := ({}^sU \cap L)L_K$ is the standard parabolic
subgroup of $L=L_J$ corresponding to $K \subseteq J$ with Levi
decomposition $Q_K = ({}^sU \cap L) \rtimes L_K$ where
{\allowdisplaybreaks
\begin{eqnarray} \label{eq:QK}
{}^sU \cap L & = & \{\begin{bmatrix}1 &   &  \\  & \mbu_{n-2}(\mbv) &
   \\  &   & 1\end{bmatrix} \mid \mbv \in \F_q^{n-4}\} = U_{n-2} \ \text{and}\\
L_K & = & \{ \begin{bmatrix}a &   &   &   &  \\  & b &   &   &  \\
  &   & \mbx &   &  \\  &   &   & b^{-1} &  \\  &   &   &   &
a^{-1}\end{bmatrix} \mid a, b \in \F_q^\times, \mbx \in \SO_{n-4}(q)\}.
\end{eqnarray}}
Thus
\begin{equation} \label{eq:QK_expl}
Q_K = \{ \begin{bmatrix}a &   &   &   &  \\  & b & * & * &  \\
  &   & \mbx & * &  \\  &   &   & b^{-1} &  \\  &   &   &   &
a^{-1}\end{bmatrix} \mid a, b \in \F_q^\times, \mbx \in \SO_{n-4}(q)\} = A
\times P_{n-2}.
\end{equation}
Setting 
\begin{eqnarray} \label{eq:L_K}
A_{n-2} & := & {}^sA = \begin{bmatrix}1 &   &   &   &  \\  & b &   &   &  \\
  &   & I_{n-4} &   &  \\  &   &   & b^{-1} &  \\  &   &   &   &
1\end{bmatrix} \mid b \in \F_q^\times\} \cong \F_q^\times \ \text{and}\\
\tilde{L}_{n-2}' & := & \{\begin{bmatrix}I_2 &   &  \\  & \mbx &  \\  &   &
I_2\end{bmatrix} \mid \mbx \in \SO_{n-4}(q)\} \cong \SO_{n-4}(q)
\end{eqnarray}
we have $L_K = A \times A_{n-2} \times L_{n-2}'$.

\item[(d)] We set $R := ({}^sU \cap U)({}^sL \cap U)$. We have:
\begin{enumerate}
\item[(i)] ${}^sL \cap U = \{\mbu_n(\mbv) \mid \mbv \in \F_q^{n-2},
  v_{m-1}=v_{m-1}'=0\}$,

\item[(ii)] ${}^sL \cap L =L_K$,

\item[(iii)] ${^s}U\cap U = \{\mbu_n(\begin{bmatrix} v_{m-1} & 0 & \dots &
  0 \end{bmatrix}^\tr) \mid v_{m-1} \in \F_q\}$ and $|{^s}U\cap U|=q$,

\item[(iv)] ${}^sP \cap P = RQ_K$,

\item[(v)] $R = \{\mbu_n(\mbv) \mid \mbv \in \F_q^{n-2}, v_{m-1}'=0\}$ and
  $[U:R] = q$,

\item[(vi)] $[P_K:RQ_K] = q$ where $P_K=UQ_K$ is the parabolic subgroup of
  $G=G_n$ corresponding to $K \subseteq \{s_1, \dots, s_m\}$,

\item[(vii)] $Q_K = A \times P_{n-2}$.
\end{enumerate}
\end{enumerate}
\end{lemma}
\begin{proof}
(a)~follows from \cite[Lemma~3.1 (a)]{AnHissUnipotent} since the Weyl
groups of type $B_m$ and $C_m$ are canonically isomorphic.

\medskip

\noindent (b)~follows from a straightforward calculation.

\medskip

\noindent (c)~It follows from~(\ref{eq:s_1}), (\ref{eq:s_j}) that $s$
commutes with $s_1, \dots, s_{m-2}$ and that $s s_{m-1} s \not\in J$.
Thus $K = \{s_1, \dots, s_{m-2}\}$. Now \cite[Theorem~2.8.7 (a) and
Proposition~2.8.9]{Carter} imply the remaining statements in (c). 

\medskip

\noindent (d)~follows from elementary matrix calculations.
\end{proof}

For $m\ge3$ we set $\tilde{P}_{n-4} := U_{n-4} \tilde{L}_{n-4}$ where 
\begin{eqnarray*}
\tilde{L}_{n-4} & := & \left\{\begin{bmatrix}a & & & & & & \\ & a & & & &
&\\ & & a & & & & \\ & & &
  \mathbf{x} & & & \\ & & & & a^{-1} && \\ & & & & & a^{-1}
  &\\&&&&&&a^{-1} \end{bmatrix} \mid \mathbf{x} \in \SO_{n-6}(q), a \in
\F_q^\times \right\} \ \text{and}
\end{eqnarray*}
\begin{eqnarray*}
U_{n-4} & := & \left\{\begin{bmatrix} I_2 & & \\ 
  & \mathbf{u}_{n-4}(\mathbf{v}) & \\ & & I_2
 \end{bmatrix} \mid \mathbf{v} \in \F_q^{n-6} \right\}.
\end{eqnarray*}
Furthermore we set $r := s_{m-1}$,
\begin{eqnarray*}
A_{n,n-2} & := & \left\{\begin{bmatrix}aI_2 &   &  \\  & I_{n-4} &
 \\  &   & aI_2\end{bmatrix} \mid a \in \F_q^\times \right\},\\
A_{n-4} & := & {}^rA_{n-2} = \left\{\begin{bmatrix}1 & & & & & & \\ & 1 & & & &
&\\ & & a & & & & \\ & & &
  I_{n-6} & & & \\ & & & & a^{-1} && \\ & & & & & 1
  &\\&&&&&&1 \end{bmatrix} \mid a \in \F_q^\times \right\}
\end{eqnarray*}
and additionally
{\allowdisplaybreaks
\begin{eqnarray*}
L_{n-2}' & := & \left\{\begin{bmatrix} 1 & & & &
\\ & 1 & & & \\ & & \mathbf{x} & & \\ & & & 1 & \\ & & & & 1
  \end{bmatrix} \mid \mathbf{x} \in \SO_{n-4}(q) \right\},\\
L_{n-4}' & := & \left\{\begin{bmatrix}1 & & & & & & \\ & 1 & & & &
&\\ & & 1 & & & & \\ & & &
  \mathbf{x} & & & \\ & & & & 1 && \\ & & & & & 1
  &\\&&&&&&1 \end{bmatrix} \mid \mathbf{x} \in \SO_{n-6}(q) \right\},
\end{eqnarray*}
\begin{eqnarray*}
Q_K' & := & A \times U_{n-2} \tilde{P}_{n-4}\\
& = & 
\left\{\begin{bmatrix}a &   &   &   &   &   &  \\ & b & * & * & * & * & \\
 & & b & * & * & * &  \\ & & & \mathbf{x} & * & * &  \\ & & & & b^{-1}
& * & \\ & & & & & b^{-1} & \\&&&&&&a^{-1} \end{bmatrix} \mid \mathbf{x} \in
\SO_{n-6}(q), a, b \in  \F_q^\times \right\} \le Q_K.
\end{eqnarray*}}

\begin{lemma} \label{la:intersec_RQK}
Let $m\ge3$. With the above notation we have 
\[
{}^{rs}(RQ_K') \cap U\tilde{P}_{n-2} = ({}^rR)Y
\]
where
\[
Y = ({}^rU_{n-2} \cap U_{n-2}) ({}^rL_{n-2} \cap U_{n-2}) ({}^rU_{n-2}
\cap L_{n-2}) (A_{n,n-2} \times A_{n-4} \times L_{n-4}').
\]
Furthermore $A \times Y = A \times ({}^rP_{n-2} \cap P_{n-2})$.
\end{lemma}
\begin{proof}
We have $\tilde{P}_{n-2} = A_{n,n-2}U_{n-2}L_{n-2}'$ and 
$Q_K' = AU_{n-2} \tilde{P}_{n-4}$. Furthermore 
\[
{}^{sr}A_{n,n-2} = \left\{\begin{bmatrix}1 & & & & & & \\ & a & & & &
&\\ & & a & & & & \\ & & &
  I_{n-6} & & & \\ & & & & a^{-1} && \\ & & & & & a^{-1}
  &\\&&&&&&1 \end{bmatrix} \mid a \in \F_q^\times \right\}
\]
and thus $A\tilde{P}_{n-4} = A ({}^{sr}A_{n,n-2}) U_{n-4} L_{n-4}'$.
Since $A$ centralizes $U_{n-2}$ we have
\[
Q_K' = A U_{n-2} A ({}^{sr}A_{n,n-2}) U_{n-4} L_{n-4}'
= A U_{n-2} ({}^{sr}A_{n,n-2}) U_{n-4} L_{n-4}'.
\]
Using that $A$ centralizes $U_{n-2}$, that $s$ normalizes $R U_{n-2}$,
that $A_{n-2} = {}^sA$, that $s$ centralizes $U_{n-4} L_{n-4}'$ and
that $A_{n-2}$ and ${}^rA_{n,n-2}$ centralize each other and normalize
$U_{n-2}$, we get
\begin{eqnarray*}
{}^{rs}(RQ_K') & = & {}^{rs}(R A U_{n-2} ({}^{sr}A_{n,n-2}) U_{n-4} L_{n-4}')\\
& = & {}^{rs}(R U_{n-2}) \ ({}^{rs}A) \ A_{n,n-2} \ {}^{rs}(U_{n-4} L_{n-4}')\\
& = & {}^r(R U_{n-2}) \ ({}^rA_{n-2}) \ {}^r({}^rA_{n,n-2}) \ {}^r(U_{n-4} L_{n-4}')\\
& = & ({}^rR) \ A_{n,n-2} \ {}^r(A_{n-2} U_{n-2} U_{n-4} L_{n-4}').
\end{eqnarray*}
An elementary matrix calculation shows that
\[
{}^rR, A_{n,n-2} \subseteq U\tilde{P}_{n-2} \quad \text{and} \quad
{}^r(A_{n-2} U_{n-2} U_{n-4} L_{n-4}')
\subseteq L_n'.
\]
Hence
\begin{eqnarray*}
{}^{rs}(RQ_K') \cap U\tilde{P}_{n-2} & = & ({}^rR) \ A_{n,n-2}
[{}^r(A_{n-2} U_{n-2} U_{n-4} L_{n-4}') \cap L_n' \cap U\tilde{P}_{n-2}]\\
& = & ({}^rR) \ A_{n,n-2}
[{}^r(A_{n-2} U_{n-2} U_{n-4} L_{n-4}') \cap (L_n' \cap U\tilde{P}_{n-2})]\\
& = & ({}^rR) \ A_{n,n-2}
[{}^r(A_{n-2} U_{n-2} U_{n-4} L_{n-4}') \cap U_{n-2} L_{n-2}']\\
& = & ({}^rR) Y,
\end{eqnarray*}
where $Y := A_{n,n-2}[U_{n-2} L_{n-2}' \cap {}^r(A_{n-2} U_{n-2} U_{n-4} L_{n-4}')]$.
By Lemma~\ref{la:intersec_par}, (c)~applied to $G_{n-2}$ we have
\[
{}^rP_{n-2} \cap P_{n-2} = 
({}^rU_{n-2} \cap U_{n-2}) 
({}^rL_{n-2} \cap U_{n-2}) 
({}^rU_{n-2} \cap L_{n-2})  
(A_{n-2} \times A_{n-4} \times L_{n-4}')
\]
where
\begin{gather*}
({}^rU_{n-2} \cap U_{n-2}) ({}^rL_{n-2} \cap U_{n-2}) = \{
\begin{bmatrix}1 &   &   &   &  &   &   &  \\
 & 1 & . & * & \dots & * & * & \\
 &  &&&&& * & \\
 &  &   && I_{n-4} && \vdots & \\
 &  &&&&& * & \\
 &  &&&&& . & \\
 &  &  &  &  &  & 1 & \\
 &  &  &  &  &  &  & 1
\end{bmatrix}\}
\end{gather*}
and
\begin{gather*}
({}^rU_{n-2} \cap L_{n-2}) (A_{n-2} \times A_{n-4}
  \times L_{n-4}') = \hspace{5cm}\\
\hspace*{3cm} \{\begin{bmatrix} 1 &   &   &   &   &   &  \\
  & a &   &   &   &   &  \\
  &   & b & * & * &   &  \\
  &   &   & \mbx & * &   &  \\
  &   &   &   & b^{-1} &   &  \\
  &   &   &   &   & a^{-1} &  \\
  &   &   &   &   &   & 1
\end{bmatrix} \mid \mbx \in \SO_{n-6}(q), a,b \in \F_q^\times\}.
\end{gather*}
We set 
\begin{eqnarray} \label{eq:Y'}
Y' & := & ({}^rU_{n-2} \cap U_{n-2}) \nonumber
({}^rL_{n-2} \cap U_{n-2}) ({}^rU_{n-2} \cap L_{n-2})  
(A_{n-4} \times L_{n-4}')\\
& = & \{\begin{bmatrix} 1 &   &   &   &   &   &  \\
  & 1 & . & * & * & * &  \\
  &   & b & * & * & * &  \\
  &   &   & \mbx & * & * &  \\
  &   &   &   & b^{-1} & . &  \\
  &   &   &   &   & 1 &  \\
  &   &   &   &   &   & 1
\end{bmatrix} \mid \mbx \in \SO_{n-6}(q), b \in \F_q^\times\}
\end{eqnarray}
so that ${}^rP_{n-2} \cap P_{n-2} = A_{n-2} Y'$. We can see from
(\ref{eq:Y'}) that $Y' \subseteq U_{n-2}L_{n-2}'$. An elementary matrix
calculation shows that ${}^rL_{n-2} \cap U_{n-2} = {}^rU_{n-4}$.
Using that $r$ centralizes $L_{n-4}'$, that $U_{n-4}$
normalizes~$U_{n-2}$ and that $A_{n-2}$ normalizes the
subgroup~$U_{n-2}$ and centralizes $U_{n-4}$, we get
\begin{eqnarray*}
Y' & \subseteq & {}^rU_{n-2} {}^rU_{n-4} {}^rU_{n-2} {}^rA_{n-2} {}^rL_{n-4}'
= {}^r(U_{n-2} U_{n-4} U_{n-2} A_{n-2} L_{n-4}')\\
& = & {}^r(A_{n-2} U_{n-2} U_{n-4} L_{n-4}')
\end{eqnarray*}
and thus 
\begin{equation} \label{eq:Y'subset}
Y' \le U_{n-2}L_{n-2}' \cap {}^r(A_{n-2} U_{n-2} U_{n-4} L_{n-4}'). 
\end{equation}
Because $P_{n-2}' := U_{n-2} L_{n-2}' \subseteq U_{n-2} L_{n-2} = P_{n-2}$,
we get
\begin{eqnarray*}
U_{n-2}L_{n-2}' \cap {}^r(A_{n-2} U_{n-2} U_{n-4} L_{n-4}') & \le &
P_{n-2}' \cap {}^rP_{n-2}\\
& = & (P_{n-2}' \cap P_{n-2}) \cap {}^rP_{n-2}\\
& = & P_{n-2}' \cap ({}^rP_{n-2}\cap P_{n-2})\\
& = & P_{n-2}' \cap A_{n-2} Y' = Y'.
\end{eqnarray*}
Hence $Y' = U_{n-2}L_{n-2}' \cap {}^r(A_{n-2} U_{n-2} U_{n-4} L_{n-4}')$ 
and $Y = A_{n,n-2} Y'$. Now the claim follows.
\end{proof}

\subsection{Restriction of Harish-Chandra induced characters}
\label{subsec:resHC}

Let $m\ge 1$ and let $\sigma$ be an irreducible character of $L$ 
with $A \le \ker(\sigma)$. In this section we study 
$R_L^G(\sigma) := (\infl_L^P \sigma)\smash\uparrow_P^G$, the character
of $G$ which is obtained from $\sigma$ by Harish-Chandra
induction. We proceed along the lines
of~\cite[Section~3]{AnHissUnipotent}. Suppressing notation for
inflation, Mackey's Theorem and Lemma~\ref{la:intersec_par} give us: 
\begin{eqnarray} \label{eq:LGP}
R_L^G(\sigma)\smash\downarrow^G_P & = & \sigma \ + \ \nonumber
{}^s\sigma\smash\downarrow^{{}^sP}_{{}^sP \cap P}\smash\uparrow_{{}^sP \cap P}^P
\ + \ {}^t\sigma\smash\downarrow^{{}^tP}_{{}^tP \cap P}\smash\uparrow_{{}^tP \cap P}^P\\
& = &\sigma \ + \
{}^s\sigma\smash\downarrow^{{}^sP}_{RQ_K}\smash\uparrow_{RQ_K}^P
\ + \ {}^t\sigma\smash\downarrow^{{}^tP}_{L}\smash\uparrow_{L}^P.
\end{eqnarray}
Since $RQ_K = {}^sP \cap P$ is $s$-invariant we have
${}^s\sigma\smash\downarrow^{{}^sP}_{RQ_K} =
{}^s\sigma\smash\downarrow^{{}^sP}_{{}^sRQ_K} =
{}^s(\sigma\smash\downarrow^P_{RQ_K})$. And furthermore,
$\sigma\smash\downarrow^P_{RQ_K} =
\infl_{P_{n-2}}^{RQ_K}(\sigma\smash\downarrow^{L'}_{P_{n-2}})$. 
Hence
${}^s\sigma\smash\downarrow^{{}^sP}_{RQ_K}\smash\uparrow_{RQ_K}^P$ is
a sum of characters of the form $({}^s\nu)\smash\uparrow_{RQ_K}^P$
where $\nu \in \Irr(P_{n-2})$ is considered as a character of 
$RQ_K = R(A \times P_{n-2})$ via inflation. Thus we have to determine 
$({}^s\nu)\smash\uparrow_{RQ_K}^P$ for $\nu \in \Irr(P_{n-2})$ of Type
1,0,$\pm$. In parts (d) and (e)~of the following theorem we need 
subgroups $P_{n-3}^\pm$ of $L_n^\pm$ which are defined as follows: 
We set 
\[
P_{n-3}^+ \hspace{-0.1cm} := \hspace{-0.1cm} 
\left\{\begin{bmatrix}a &   &   &   &   &   &  \\
  & b & * & .  & * & * &  \\
  &   & A & .  & B & * &  \\
  &   & .  & a & .  & .  &  \\
  &   & C & .  & D & * &  \\
  &   &   &   &   & \hspace{- .1cm} b^{-1} \hspace{- .1cm} &  \\
  &   &   &   &   &   & a
\end{bmatrix} \hspace{-0.1cm} | \hspace{-0.1cm} 
\begin{bmatrix} A & \hspace{-0.2cm} B\\ C & \hspace{-0.2cm} D\end{bmatrix}   \in \GO^+_{n-5}(q),
  a=\det\begin{bmatrix} A & \hspace{-0.2cm} B\\ C & \hspace{-0.2cm} D\end{bmatrix}, b \in \F_q^\times \right\}
\]
for $m>2$ and
\[
P_2^+ := \{
\mathrm{diag}(1,b,1,b^{-1},1)
\mid b \in \F_q^\times\} \cong \F_q^\times
\]
for $m=2$. For $m>2$ we define $P_{n-3}^-$ to be the subgroup of
$L_n^-$ consisting of all matrices $\mathbf{s}_n(\mathbf{y}, a)$ 
with  
\[
\mathbf{y} \in 
\left\{\mathbf{b}_{n-2}^{-1} \hspace{-0.1cm} 
\begin{bmatrix}b \hspace{-0.1cm} & * & . & * & *\\
 \hspace{-0.1cm} & A & . & B & *\\
 \hspace{-0.1cm} & . & a & . & .\\
 \hspace{-0.1cm} & C & . & D & *\\
 \hspace{-0.1cm} &  &  &  & b^{-1}
\end{bmatrix} \hspace{-0.1cm} \mathbf{b}_{n-2} \hspace{-0.1cm} \mid \hspace{-0.1cm} 
\begin{bmatrix} A \hspace{-0.1cm} & \hspace{-0.1cm}
  B\\ C \hspace{-0.1cm} & \hspace{-0.1cm} D\end{bmatrix}   \in
  \GO_{n-5}^-(q), a=\det\begin{bmatrix} A \hspace{-0.1cm} & \hspace{-0.1cm}
  B\\ C \hspace{-0.1cm} & \hspace{-0.1cm} D\end{bmatrix} \hspace{-0.1cm} , b \in \F_q^\times \right\}
\]
where the matrix $\mathbf{b}_{n-2}$ is defined in
Section~\ref{subsec:inertiasubgrps}. The following theorem and parts 
of its proof are analogous to~\cite[Theorem~3.3]{AnHissUnipotent}.

\begin{theorem} \label{thm:indnu}
Assume that $m \ge 2$. Let $\sigma \in \Irr(L)$ with 
$A \le \ker(\sigma)$. Then the following statements hold:
\begin{enumerate}
\item[(a)]  ${}^t\sigma\smash\downarrow^{{}^tP}_{L} =
{}^t\sigma$. In particular, ${}^t\sigma\smash\downarrow^{{}^tP}_{{}^tP
  \cap P}\smash\uparrow_{{}^tP \cap P}^P = ({}^t\sigma)\smash\uparrow_{L}^P$.

\item[(b)] Let $\nu \in \Irr(P_{n-2})$ be of Type 1. We
  view $\nu$ as an irreducible character of the groups 
  $L_K = A \times L_{n-2} = A \times A_{n-2} \times L_{n-2}'$ and 
  $RQ_K$ via inflation. Then
  \[
  ({}^s\nu)\smash\uparrow_{RQ_K}^P = {}^0\psi_{\nu} + {}^1\psi_\Sigma
  \]
  where $\Sigma := R_{L_K}^L({}^s\nu)$.

\item[(c)] Suppose that $m \ge 3$. Let $r := s_{m-1}$ and let
$\nu \in \Irr(P_{n-2})$ be of Type 0, i.e., there is 
$\nu_0 \in \Irr(P_{n-4})$ such that $\nu = {}^0\psi_{\nu_0}$. By
inflation we consider $\nu_0$ as an irreducible character of
\[
{}^rP_{n-2} \cap P_{n-2} = \{\begin{bmatrix}1&&&&&&\\& a & . & * & * & *
  &\\&& b & * & * & * &\\&&& \mbx & * & * &\\&&&&b^{-1} & .
&\\&&&&&a^{-1}&\\&&&&&&1\end{bmatrix} \mid \mbx \in \SO_{n-6}(q), a,b \in \F_q^\times\}.
\]
Since ${}^rP_{n-2} \cap P_{n-2}$ is $r$-invariant we have ${}^r\nu_0
\in \Irr({}^rP_{n-2} \cap P_{n-2})$. Then
\[
({}^s\nu)\smash\uparrow_{RQ_K}^P = {}^0\psi_\Sigma
\]
where $\Sigma := ({}^r\nu_0)\smash\uparrow_{{}^rP_{n-2} \cap P_{n-2}}^{P_{n-2}}$.

\item[(d)] Let $\nu \in \Irr(P_{n-2})$ be of Type $+$, say 
$\nu = {}^+\psi_{\vartheta_0}$ for some 
$\vartheta_0 \in \Irr(L_{n-2}^+)$, and let 
$\vartheta := R_{L_{n-2}^+}^{L_n^+}(\vartheta_0) :=
  \tilde{\vartheta}_0\smash\uparrow_{P_{n-3}^+}^{L_n^+}$ where
$\tilde{\vartheta}_0 \in \Irr(P_{n-3}^+)$ is the inflation of $\vartheta_0$
given by 
\[
\tilde{\vartheta}_0(\begin{bmatrix}
a &   &   &   &   &   &  \\
  & b & * & . & * & * &  \\
  &   & A & . & B & * &  \\
  &   & . & a & . & . &  \\
  &   & C & . & D & * &  \\
  &   &   &   &   & b^{-1} &  \\
  &   &   &   &   &   & a
\end{bmatrix}) := 
\vartheta_ (\begin{bmatrix}a &   &   &   &   \\
  & A & . & B &  \\
  & . & a & . &  \\
  & C & . & D &  \\
  &   &   &   & a\end{bmatrix}).
\]
Then $({}^s\nu)\smash\uparrow_{RQ_K}^P = {}^+\psi_\vartheta$.

\item[(e)] Let $\nu \in \Irr(P_{n-2})$ be of Type $-$, say  
$\nu = {}^-\psi_{\vartheta_0}$ for some 
$\vartheta_0 \in \Irr(L_{n-2}^-)$, and let 
$\vartheta := R_{L_{n-2}^-}^{L_n^-}(\vartheta_0) :=
  \tilde{\vartheta}_0\smash\uparrow_{P_{n-3}^-}^{L_n^-}$ where
$\tilde{\vartheta}_0 \in \Irr(P_{n-3}^-)$ is the inflation of $\vartheta_0$
given by 
\[
\tilde{\vartheta}_0(
\mathbf{s}_n(
\mathbf{b}_{n-2}^{-1} \hspace{-0,1cm} \begin{bmatrix}
b & * & . & * & \hspace{- ,1cm} *\\
  & A & . & B & \hspace{- ,1cm} *\\
  & . & a & . & \hspace{- ,1cm} .\\
  & C & . & D & \hspace{- ,1cm} *\\
  &   &   &   & \hspace{- ,1cm} b^{-1}
\end{bmatrix} \hspace{-0,1cm} \mathbf{b}_{n-2}, a) := 
\vartheta_0(\mathbf{s}_{n-2}(
\mathbf{b}_{n-4}^{-1} \hspace{-0,1cm} 
\begin{bmatrix}A & . & B\\. & a & .\\C & . & D
\end{bmatrix} \hspace{-0,1cm} \mathbf{b}_{n-4}, a)).
\]
Then $({}^s\nu)\smash\uparrow_{RQ_K}^P = {}^-\psi_\vartheta$.

\end{enumerate}
\end{theorem}
\begin{proof}
(a)~A straightforward computation shows that $t$ normalizes $L'$ and
$L$. It follows that ${}^t\sigma\smash\downarrow^{{}^tP}_{L} =
{}^t\sigma$ proving (a).

\medskip

(b)~Note that $s$ normalizes $L_K$ so that ${}^s\nu \in \Irr(L_K)$. We
first consider the special case that $A_{n-2} \le \ker(\nu)$. Then
${}^s\nu = \nu$ since $s$ centralizes $L_{n-2}'$. Let $\tilde{\nu}$ be
the irreducible character of $\tilde{P}_{n-2}$ corresponding to $\nu$
as in Section~\ref{subsec:charsPn}. By Mackey's Theorem
we have 
\begin{eqnarray*}
\langle (\hat{\lambda}^0 \cdot
\tilde{\nu})\smash\uparrow_{U\tilde{P}_{n-2}}^{UQ_K},
\nu\smash\uparrow_{RQ_K}^{UQ_K} \rangle_{UQ_K} & = & 
\langle \hat{\lambda}^0 \cdot \tilde{\nu},
\nu\smash\uparrow_{RQ_K}^{UQ_K}\smash\downarrow^{UQ_K}_{U\tilde{P}_{n-2}}
\rangle_{U\tilde{P}_{n-2}} \\
& = & \langle \hat{\lambda}^0 \cdot \tilde{\nu},
\nu\smash\downarrow^{RQ_K}_{R\tilde{P}_{n-2}}\smash\uparrow_{R\tilde{P}_{n-2}}^{U\tilde{P}_{n-2}} 
\rangle_{U\tilde{P}_{n-2}}\\
& = & \langle \hat{\lambda}^0\smash\downarrow^{U\tilde{P}_{n-2}}_{R\tilde{P}_{n-2}} \cdot \tilde{\nu},
\tilde{\nu} \rangle_{R\tilde{P}_{n-2}}\\
& = & \langle \tilde{\nu}, \tilde{\nu} \rangle_{R\tilde{P}_{n-2}} = 1
\end{eqnarray*}
since $R \le \ker(\hat{\lambda}^0)$ and 
$\hat{\lambda}^0\smash\downarrow^{U\tilde{P}_{n-2}}_{\tilde{P}_{n-2}}
= 1_{\tilde{P}_{n-2}}$. It follows that ${}^0\psi_\nu = (\hat{\lambda}^0 \cdot
\tilde{\nu})\smash\uparrow_{U\tilde{P}_{n-2}}^P$ is an
irreducible constituent of $\nu\smash\uparrow_{RQ_K}^P = ({}^s\nu)\smash\uparrow_{RQ_K}^P$.

Furthermore $\nu$ is a constituent of
$\nu\smash\uparrow_{RQ_K}^{UQ_K}$. Hence 
${}^1\psi_\Sigma = ({}^s\nu)\smash\uparrow_{UQ_K}^P =
\nu\smash\uparrow_{UQ_K}^P$ is a subcharacter of
$\nu\smash\uparrow_{RQ_K}^P = ({}^s\nu)\smash\uparrow_{RQ_K}^P$. Hence
$({}^s\nu)\smash\uparrow_{RQ_K}^P = {}^0\psi_\nu + {}^1\psi_\Sigma + (\dots)$.
Comparing degrees we get 
$({}^s\nu)\smash\uparrow_{RQ_K}^P = {}^0\psi_\nu + {}^1\psi_\Sigma$. 

Now we deal with the general case. Write 
$\nu = 1_A \boxtimes \zeta \boxtimes \nu'$ where $\zeta \in
\Irr(A_{n-2})$ and $\nu' \in \Irr(L_{n-2}')$. Then
${}^s\nu = {}^s\zeta \boxtimes 1_{A_{n-2}} \boxtimes \nu'$. 
Considering ${}^s\zeta$ as a linear character of $P$ via
inflation we obtain $({}^s\nu)\smash\uparrow_{RQ_K}^P = {}^s\zeta
\cdot ({}^0\psi_{\nu'} + {}^1\psi_{\Sigma'})$ where 
$\Sigma' := R_{L_K}^L(\nu')$. Hence we get
\[
({}^s\nu)\smash\uparrow_{RQ_K}^P = {}^s\zeta
\cdot {}^0\psi_{\nu'} + {}^s\zeta
\cdot {}^1\psi_{\Sigma'} = {}^0\psi_{\nu} + {}^1\psi_\Sigma.
\]

\medskip


(c)~By assumption, $\nu \in \Irr(P_{n-2})$ is of Type 0.
Let $\lambda^0_{n-2}$ be the irreducible character of $U_{n-2}$
analogous to $\lambda^0 \in \Irr(U)$ and let $\hat{\lambda}^0_{n-2}$
be the extension of $\lambda_{n-2}^0$ to $U_{n-2}\tilde{P}_{n-4}$ such
that
$\hat{\lambda}^0_{n-2}\smash\downarrow^{U_{n-2}\tilde{P}_{n-4}}_{\tilde{P}_{n-4}}
  = 1_{\tilde{P}_{n-4}}$. By definition
\begin{equation} \label{eq:type0}
\nu = (\hat{\lambda}_{n-2}^0 \cdot
\infl_{\tilde{P}_{n-4}}^{U_{n-2}\tilde{P}_{n-4}}(\tilde{\nu}_0))\smash\uparrow_{U_{n-2}\tilde{P}_{n-4}}^{P_{n-2}}. 
\end{equation}
As before we set $Q_K' := A \times U_{n-2}\tilde{P}_{n-4} \le Q_K$. 
We inflate the characters in (\ref{eq:type0}) over the normal subgroup
$RA$ to $RQ_K'$ and $RQ_K$, respectively. Suppressing the symbols for
inflation we obtain:
\[
\nu = (\hat{\lambda}_{n-2}^0 \cdot \tilde{\nu}_0)\smash\uparrow_{RQ_K'}^{RQ_K}. 
\]
Since $r \in P$ and $RQ_K$ is $s$-invariant we have
\begin{eqnarray} \label{eq:indconjr}
({}^s\nu)\smash\uparrow_{RQ_K}^P & = &
  {}^r(({}^s\nu)\smash\uparrow_{RQ_K}^P) = \nonumber
  {}^r(({}^s((\hat{\lambda}_{n-2}^0 \cdot
  \tilde{\nu}_0)\smash\uparrow_{RQ_K'}^{RQ_K}))\smash\uparrow_{RQ_K}^P)\\
  & = & {}^r(({}^s\hat{\lambda}_{n-2}^0 \cdot \nonumber
  {}^s\tilde{\nu}_0)\smash\uparrow_{{}^s(RQ_K')}^{RQ_K}\smash\uparrow_{RQ_K}^P)
  = {}^r(({}^s\hat{\lambda}_{n-2}^0 \cdot
  {}^s\tilde{\nu}_0)\smash\uparrow_{{}^s(RQ_K')}^P)\\
  & = & ({}^{rs}\hat{\lambda}_{n-2}^0 \cdot
  {}^{rs}\tilde{\nu}_0)\smash\uparrow_{{}^{rs}(RQ_K')}^P.
\end{eqnarray}
We set $\widetilde{\Sigma} := (1_A \boxtimes
\Sigma)\smash\downarrow^{A \times P_{n-2}}_{\tilde{P}_{n-2}}$ so that
${}^0\psi_\Sigma = (\hat{\lambda}^0
\cdot\infl^{U\tilde{P}_{n-2}}_{\tilde{P}_{n-2}}(\widetilde{\Sigma}))\smash\uparrow_{U\tilde{P}_{n-2}}^P$.

By Lemma~\ref{la:intersec_RQK} we have 
\begin{equation} \label{eq:intlem}
U\tilde{P}_{n-2} \cap {}^{rs}(RQ_K') = ({}^rR)Y
\end{equation}
with $Y \le \tilde{P}_{n-2}$ and $A \times Y = A \times ({}^rP_{n-2} \cap P_{n-2})$.
Now
\[
{}^{rs}U_{n-2} = \{\begin{bmatrix}
 1 & * & . & * & * & * & . & * & *\\
&1&.&.&.&.&.&.&*\\&&1&.&.&.&.&.&.\\&&&&&&.&.&*\\
&&&&I_{n-6}&&.&.&*\\&&&&&&.&.&*\\&&&&&&1&.&.\\&&&&&&&1&*\\
&&&&&&&&1\end{bmatrix}\} \le U. 
\]
Let $v = \begin{bmatrix}v_{m-2}, \dots, v_1, v_0, v_1',
  \dots, v_{m-2}'\end{bmatrix}^\tr \in \F_q^{n-4}$. Then
\[
{}^{rs}\mbu_{n-2}(v) = \mbu_n(\begin{bmatrix}*, 0,
    v_{m-2}'\end{bmatrix}^\tr)
\]
and thus ${}^{rs}\hat{\lambda}_{n-2}^0({}^{rs}\mbu_{n-2}(v)) =
\hat{\lambda}_{n-2}^0(\mbu_{n-2}(v)) = \xi(v_{m-2}') =
\lambda^0({}^{rs}\mbu_{n-2}(v))$. Hence the characters
${}^{rs}\hat{\lambda}_{n-2}^0$ and $\lambda^0$ coincide on
${}^{rs}U_{n-2}$. As above, we consider 
$\hat{\lambda}_{n-2}^0 \cdot \tilde{\nu}_0$ and $\tilde{\nu}_0$ as
irreducible characters of $RQ_K'$ via inflation. Thus 
${}^{rs}\hat{\lambda}_{n-2}^0 \cdot {}^{rs}\tilde{\nu}_0 \in
\Irr({}^{rs}(RQ_K'))$. By restriction via (\ref{eq:intlem}) we can
view ${}^{rs}\hat{\lambda}_{n-2}^0 \cdot {}^{rs}\tilde{\nu}_0$ as a
character of $({}^rR)Y \le UY$. An elementary matrix calculation shows
that ${}^sR \subseteq RU_{n-2} \subseteq \ker(\tilde{\nu}_0)$. It
follows that 
\[
{}^rR = \{\begin{bmatrix}1 & * & . & * & * & * & * & * & *\\
&1&.&.&.&.&.&.&*\\&&1&.&.&.&.&.&*\\&&&&&&.&.&*\\
&&&&I_{n-6}&&.&.&*\\&&&&&&.&.&*\\&&&&&&1&.&.\\&&&&&&&1&*\\
&&&&&&&&1\end{bmatrix}\} \le \ker({}^{rs}\tilde{\nu}_.).
\]
Since $Y$ normalizes ${}^rR$ and $U$ the character
${}^{rs}\tilde{\nu}_0$ has a unique extension to $UY$ 
with $U$ in its kernel and we denote this extension also by
${}^{rs}\tilde{\nu}_0$. An elementary matrix calculation shows that 
\[
\{\begin{bmatrix}1 & . & . & . & . & . & * & . & *\\
&1&.&.&.&.&.&.&.\\&&1&.&.&.&.&.&*\\&&&&&&.&.&.\\
&&&&I_{n-6}&&.&.&.\\&&&&&&.&.&.\\&&&&&&1&.&.\\&&&&&&&1&.\\
&&&&&&&&1\end{bmatrix}^{rs}\} \le \ker(\hat{\lambda}_{n-2}^.)
\]
and it follows that ${}^{rs}\hat{\lambda}_{n-2}^0$ coincides with
$\lambda^0$ and $\hat{\lambda}^0$ on ${}^rR$. Similarly, we see that 
$Y^{rs} \le \ker(\hat{\lambda}_{n-2}^0)$, hence
$Y \le \ker({}^{rs}\hat{\lambda}_{n-2}^0), \ker(\hat{\lambda}^0)$.
Thus ${}^{rs}\hat{\lambda}_{n-2}^0$ coincides with $\hat{\lambda}^0$
on~$Y$, too. Therefore, $\hat{\lambda}^0$ considered as an irreducible
character of $UY$ via restriction is an extension of 
${}^{rs}\hat{\lambda}_{n-2}^0$ to $UY$ and $\hat{\lambda}^0 \cdot
{}^{rs}\tilde{\nu}_0$ is an extension of 
${}^{rs}\hat{\lambda}_{n-2}^0 \cdot {}^{rs}\tilde{\nu}_0$ from
$({}^rR)Y$ to $UY$. Thus $\hat{\lambda}^0 \cdot
{}^{rs}\tilde{\nu}_0$ is a subcharacter of
$({}^{rs}\hat{\lambda}_{n-2}^0 \cdot
{}^{rs}\tilde{\nu}_0)\smash\uparrow_{({}^rR)Y}^{UY}$
and therefore $(\hat{\lambda}^0 \cdot
{}^{rs}\tilde{\nu}_0)\smash\uparrow_{UY}^{U\tilde{P}_{n-2}}$ 
is a subcharacter of $({}^{rs}\hat{\lambda}_{n-2}^0 \cdot
{}^{rs}\tilde{\nu}_0)\smash\uparrow_{({}^rR)Y}^{U\tilde{P}_{n-2}}$.
The isomorphism
\[
\tilde{P}_{n-2} \to P_{n-2}, 
\begin{bmatrix}
a& & &     & \\ &a&*&*    & \\ & &x&*    & \\
 & & &a^{-1}& \\ & & &     &a^{-1}\end{bmatrix} \mapsto
\begin{bmatrix}
1& & &     & \\ &a&*&*    & \\ & &x&*    & \\
 & & &a^{-1}& \\ & & &     &1\end{bmatrix}
\]
maps $Y$ onto ${}^rP_{n-2} \cap P_{n-2}$ and hence induces bijections
$\Irr(P_{n-2}) \to \Irr(\tilde{P}_{n-2})$ and 
$\Irr({}^rP_{n-2} \cap P_{n-2}) \to \Irr(Y)$ which we both
denote by $\phi$. We extend $\phi$ additively to non-irreducible
characters of these groups. An elementary matrix calculation shows
that $\widetilde{\Sigma} = \phi(\Sigma)$ and 
${}^{rs}\tilde{\nu}_0 = \phi({}^r\nu_0)$. The functoriality of
induction implies that $\phi$ commutes with induction, hence
\[
\widetilde{\Sigma} = \phi(\Sigma) =
\phi(({}^r\nu_0)\smash\uparrow_{{}^rP_{n-2} \cap P_{n-2}}^{P_{n-2}}) = 
\phi({}^r\nu_0)\smash\uparrow_Y^{\tilde{P}_{n-2}} = 
({}^{rs}\tilde{\nu}_0)\smash\uparrow_Y^{\tilde{P}_{n-2}}.
\]
Again suppressing the notation for inflation we get that 
\[
\hat{\lambda}^0 \cdot \widetilde{\Sigma} = \hat{\lambda}^0 \cdot
({}^{rs}\tilde{\nu}_0)\smash\uparrow_{UY}^{U\tilde{P}_{n-2}} = 
(\hat{\lambda}^0 \cdot {}^{rs}\tilde{\nu}_0)\smash\uparrow_{UY}^{U\tilde{P}_{n-2}}
\] 
is a subcharacter of
\[
({}^{rs}\hat{\lambda}_{n-2}^0 \cdot
{}^{rs}\tilde{\nu}_0)\smash\uparrow_{({}^rR)Y}^{U\tilde{P}_{n-2}} =
({}^{rs}\hat{\lambda}_{n-2}^0 \cdot 
{}^{rs}\tilde{\nu}_0)\smash\downarrow^{{}^{rs}(RQ_K')}_{{}^{rs}(RQ_K')
  \cap U\tilde{P}_{n-2}}\smash\uparrow_{{}^{rs}(RQ_K') \cap
  U\tilde{P}_{n-2}}^{U\tilde{P}_{n-2}}.
\]
It follows from Mackey's Theorem and (\ref{eq:indconjr}) that 
$\hat{\lambda}^0 \cdot \widetilde{\Sigma}$ is a subcharacter of
\[
({}^{rs}\hat{\lambda}_{n-2}^0 \cdot 
{}^{rs}\tilde{\nu}_0)\smash\uparrow_{{}^{rs}(RQ_K')}^P\smash\downarrow^P_{U\tilde{P}_{n-2}}
= ({}^s\nu)\smash\uparrow_{RQ_K}^P\smash\downarrow^P_{U\tilde{P}_{n-2}}.
\]
Write $\hat{\lambda}^0 \cdot \widetilde{\Sigma} = \sum_i m_i\chi_i$
with $\chi_i \in \Irr(U\tilde{P}_{n-2})$, hence 
\[
({}^s\nu)\smash\uparrow_{RQ_K}^P\smash\downarrow^P_{U\tilde{P}_{n-2}}
= \sum_i m_i\chi_i + (\dots)
\]
and thus
$\langle\chi_i\smash\uparrow_{U\tilde{P}_{n-2}}^P,
({}^s\nu)\smash\uparrow_{RQ_K}^P \rangle_P = \langle\chi_i,
({}^s\nu)\smash\uparrow_{RQ_K}^P\smash\downarrow^P_{U\tilde{P}_{n-2}}\rangle_{U\tilde{P}_{n-2}}
\ge m_i$. Note that the characters
$\chi_i\smash\uparrow_{U\tilde{P}_{n-2}}^P$ are irreducible by
Clifford theory. Hence, ${}^0\psi_\Sigma = (\hat{\lambda}^0 \cdot 
\widetilde{\Sigma})\smash\uparrow_{U\tilde{P}_{n-2}}^P$ is a
subcharacter of $({}^s\nu)\smash\uparrow_{RQ_K}^P$. We have
\begin{eqnarray*}
({}^s\nu)\smash\uparrow_{RQ_K}^P(1) & = & [P:RQ_K] \cdot \nu(1) =
  \frac{|U||L|}{|R||Q_K|} (q^{n-5}-1)\nu_0(1)\\
& = & q \cdot \frac{(q-1) \cdot |\SO_{n-2}(q)|}{(q-1) \cdot |P_{n-2}|}
  (q^{2m-4}-1)\nu_0(1)\\
& = & q \cdot \frac{(q^{2m-2}-1)(q^{2m-4}-1)}{q-1} \cdot \nu_0(1)
\end{eqnarray*}
and
\begin{eqnarray*}
{}^0\psi_\Sigma & = & (q^{n-3}-1) \cdot \Sigma(1) = (q^{n-3}-1) \cdot
[P_{n-2} : {}^rP_{n-2} \cap P_{n-2}] \cdot \nu_0(1)\\
& = & (q^{n-3}-1) \cdot \frac{|P_{n-2}|}{q^{n-5} \cdot (q-1) \cdot
  |P_{n-4}|} \cdot \nu_0(1)\\
& = & q \cdot \frac{(q^{2m-2}-1)(q^{2m-4}-1)}{q-1} \cdot \nu_0(1),
\end{eqnarray*}
thus $({}^s\nu)\smash\uparrow_{RQ_K}^P = {}^0\psi_\Sigma$ proving (c).

\medskip

(d)~We denote the inflation of the character $\vartheta_0 \cdot
\hat{\lambda}_{n-2}^+ \in \Irr(U_{n-2}L_{n-2}^+)$ to the group 
$R(A \times U_{n-2}L_{n-2}^+)$ also by $\vartheta_0 \cdot
\hat{\lambda}_{n-2}^+ $. Hence 
\[
{}^s\vartheta_0 \cdot {}^s\hat{\lambda}_{n-2}^+ \in \Irr({}^sR(A_{n-2} \times {}^s(U_{n-2}L_{n-2}^+))).
\]
An elementary matrix calculation shows that
$RP_{n-3}^+ \le {}^sR(A_{n-2} \times {}^s(U_{n-2}L_{n-2}^+))$.

The proof now proceeds in several steps.

\medskip

\textbf{Step 1:} We have $({}^s\vartheta_0 \cdot
{}^s\hat{\lambda}_{n-2}^+)\smash\downarrow^{{}^sR(A_{n-2}  
\times {}^s(U_{n-2}L_{n-2}^+))}_{RP_{n-3}^+} = (\tilde{\vartheta}_0 \cdot
\hat{\lambda}^+)\smash\downarrow^{UP_{n-3}^+}_{RP_{n-3}^+}$.

\textbf{Proof of Step 1:} Let $\mbu_n(\begin{bmatrix}v_{m-1}, 
  \dots, v_0, \dots, v_{m-2}', 0\end{bmatrix}^\tr) \in R$. Then
\begin{eqnarray*}
&&({}^s\vartheta_0 \cdot
{}^s\hat{\lambda}_{n-2}^+)(\mbu_n(\begin{bmatrix}v_{m-1}, 
  \dots, v_0, \dots, v_{m-2}', 0\end{bmatrix}^\tr))\\
& = & (\vartheta_0 \cdot \hat{\lambda}_{n-2}^+)(\mbu_n(\begin{bmatrix}v_{m-1}, 
  \dots, v_0, \dots, v_{m-2}', 0\end{bmatrix}^\tr)^s)\\
& = & \vartheta_0(1) \xi(v_0) = \vartheta_0(1) \cdot 
\hat{\lambda}^+(\mbu_n(\begin{bmatrix}v_{m-1}, 
  \dots, v_0, \dots, v_{m-2}', 0\end{bmatrix}^\tr)).
\end{eqnarray*}
It follows that $({}^s\vartheta_0 \cdot
{}^s\hat{\lambda}_{n-2}^+)\smash\downarrow^{{}^sR(A_{n-2}  
\times {}^s(U_{n-2}L_{n-2}^+))}_{RP_{n-3}^+} = \varphi \cdot 
(\hat{\lambda}^+\smash\downarrow^{UP_{n-3}^+}_{RP_{n-3}^+})$ where
$\varphi$ is the inflation of some character of $P_{n-3}^+$ to
$RP_{n-3}^+$. Note that $\varphi$ is uniquely determined by its
restriction to $P_{n-3}^+$. We have
\begin{gather*}
\varphi(\begin{bmatrix}a &   &   &   &   &   &  \\
  & b & * & . & * & * &  \\
  &   & A & . & B & * &  \\
  &   & . & a & . & . &  \\
  &   & C & . & D & * &  \\
  &   &   &   &   & b^{-1} &  \\
  &   &   &   &   &   & a
\end{bmatrix}) = ({}^s\vartheta_  \cdot
{}^s\hat{\lambda}_{n-2}^+)(
\begin{bmatrix}
a &   &   &   &   &   &  \\
  & b & * & . & * & * &  \\
  &   & A & . & B & * &  \\
  &   & . & a & . & . &  \\
  &   & C & . & D & * &  \\
  &   &   &   &   & b^{-1} &  \\
  &   &   &   &   &   & a
\end{bmatrix})\\
= (\vartheta_  \cdot \hat{\lambda}_{n-2}^+)( \hspace{- .1cm}
\begin{bmatrix}
a &   &   &   &   &   &  \\
  & b & * & . & * & * &  \\
  &   & A & . & B & * &  \\
  &   & . & a & . & . &  \\
  &   & C & . & D & * &  \\
  &   &   &   &   & b^{-1} &  \\
  &   &   &   &   &   & a
\end{bmatrix}^s \hspace{- .1cm} )
= (\vartheta_  \cdot \hat{\lambda}_{n-2}^+)( \hspace{- .1cm} 
\begin{bmatrix}
b & . & * & . & * & . & *\\
  & a & . & . & . & . & .\\
  &   & A & . & B & . & *\\
  &   & . & a & . & . & .\\
  &   & C & . & D & . & *\\
  &   &   &   &   & a & .\\
  &   &   &   &   &   & b^{-1}
\end{bmatrix} \hspace{- .1cm} )
\end{gather*}
\begin{gather*}
= \vartheta_0(\begin{bmatrix}
a &   &   &   &   \\
  & A & . & B &  \\
  & . & a & . &  \\
  & C & . & D &  \\
  &   &   &   & a\end{bmatrix})
= \tilde{\vartheta}_0(
\begin{bmatrix}
a &   &   &   &   &   &  \\
  & b & * & . & * & * &  \\
  &   & A & . & B & * &  \\
  &   & . & a & . & . &  \\
  &   & C & . & D & * &  \\
  &   &   &   &   & b^{-1} &  \\
  &   &   &   &   &   & a
\end{bmatrix}),
\end{gather*}
thus $\varphi\smash\downarrow^{RP_{n-3}^+}_{P_{n-3}^+} =
\tilde{\vartheta}_0\smash\downarrow^{RP_{n-3}^+}_{P_{n-3}^+}$ and then 
\[
({}^s\vartheta_0 \cdot
{}^s\hat{\lambda}_{n-2}^+)\smash\downarrow^{{}^sR(A_{n-2}  
\times {}^s(U_{n-2}L_{n-2}^+))}_{RP_{n-3}^+} = \tilde{\vartheta}_0 \cdot 
(\hat{\lambda}^+\smash\downarrow^{UP_{n-3}^+}_{RP_{n-3}^+}) = 
(\tilde{\vartheta}_0 \cdot \hat{\lambda}^+)\smash\downarrow^{UP_{n-3}^+}_{RP_{n-3}^+}.
\]

\medskip

\textbf{Step 2:} The character $\vartheta \cdot \hat{\lambda}^+$ is a
subcharacter of ${}^s\nu\smash\uparrow_{RQ_K}^P\smash\downarrow^P_{UL_n^+}$.

\textbf{Proof of Step 2:} It follows from Step~1 that 
$\tilde{\vartheta}_0 \cdot \hat{\lambda}^+$ is an irreducible
constituent of $({}^s\vartheta_0 \cdot
{}^s\hat{\lambda}_{n-2}^+)\smash\downarrow^{{}^sR(A_{n-2}  
\times {}^s(U_{n-2}L_{n-2}^+))}_{RP_{n-3}^+}\smash\uparrow_{RP_{n-3}^+}^{UP_{n-3}^+}$. 
Thus 
\[
(\tilde{\vartheta}_0 \cdot \hat{\lambda}^+)\smash\uparrow_{UP_{n-3}^+}^{UL_n^+}
= \tilde{\vartheta}_0\smash\uparrow_{UP_{n-3}^+}^{UL_n^+} \cdot \hat{\lambda}^+
= \vartheta \cdot \hat{\lambda}^+
\]
is a subcharacter of $({}^s\vartheta_0 \cdot
{}^s\hat{\lambda}_{n-2}^+)\smash\downarrow^{{}^sR(A_{n-2}  
\times {}^s(U_{n-2}L_{n-2}^+))}_{RP_{n-3}^+}\smash\uparrow_{RP_{n-3}^+}^{UL_n^+}$. 
By construction, the character $\nu \in \Irr(RQ_K)$ is given by
$\nu = (\vartheta_0 \cdot \hat{\lambda}_{n-2}^+)\smash\uparrow_{R(A \times
  U_{n-2}L_{n-2}^+)}^{RQ_K}$ and hence 
\begin{equation} \label{eq:constnus}
{}^s\nu = ({}^s\vartheta_0 \cdot {}^s\hat{\lambda}_{n-2}^+)\smash\uparrow_{{}^sR({}^sA \times
  {}^s(U_{n-2}L_{n-2}^+))}^{RQ_K} = ({}^s\vartheta_0 \cdot
{}^s\hat{\lambda}_{n-2}^+)\smash\uparrow_{{}^sR(A_{n-2} \times {}^s(U_{n-2}L_{n-2}^+))}^{RQ_K}.
\end{equation}
It follows from (\ref{eq:constnus}) that ${}^s\vartheta_0 \cdot {}^s\hat{\lambda}_{n-2}^+$
is a constituent of ${}^s\nu\smash\downarrow^{RQ_K}_{{}^sR(A_{n-2} \times {}^s(U_{n-2}L_{n-2}^+))}$.
Hence $({}^s\vartheta_0 \cdot
{}^s\hat{\lambda}_{n-2}^+)\smash\downarrow^{{}^sR(A_{n-2} \times
  {}^s(U_{n-2}L_{n-2}^+))}_{RP_{n-3}^+}\smash\uparrow_{RP_{n-3}^+}^{UL_n^+}$ is a subcharacter of
${}^s\nu\smash\downarrow^{RQ_K}_{RP_{n-3}^+}\smash\uparrow_{RP_{n-3}^+}^{UL_n^+}$
and from the beginning of the proof of Step~2 we see that 
$\vartheta \cdot \hat{\lambda}^+$ is a subcharacter of
${}^s\nu\smash\downarrow^{RQ_K}_{RP_{n-3}^+}\smash\uparrow_{RP_{n-3}^+}^{UL_n^+}$. 
An elementary matrix calculation shows that
\[
RQ_K \cap UL_n^+ = RP_{n-3}^+.
\]
Thus Mackey's Theorem implies that the character
${}^s\nu\smash\downarrow^{RQ_K}_{RP_{n-3}^+}\smash\uparrow_{RP_{n-3}^+}^{UL_n^+}$
is a subcharacter of 
${}^s\nu\smash\uparrow_{RQ_K}^P\smash\downarrow^P_{UL_n^+}$ 
and we can
conclude that $\vartheta \cdot \hat{\lambda}^+$ is a subcharacter of 
${}^s\nu\smash\uparrow_{RQ_K}^P\smash\downarrow^P_{UL_n^+}$.

\medskip

\textbf{Step 3:} We have $({}^s\nu)\smash\uparrow_{RQ_K}^P = {}^+\psi_\vartheta$.

\textbf{Proof of Step 3:} Write 
$\vartheta \cdot \hat{\lambda}^+ = \sum_i m_i\chi_i$
with $\chi_i \in \Irr(UL_n^+)$. By Step~2 we have
\[
{}^s\nu\smash\uparrow_{RQ_K}^P\smash\downarrow^P_{UL_n^+}
= \sum_i m_i\chi_i + (\dots)
\]
and thus
$\langle\chi_i\smash\uparrow_{UL_n^+}^P,
({}^s\nu)\smash\uparrow_{RQ_K}^P \rangle_P = \langle\chi_i,
({}^s\nu)\smash\uparrow_{RQ_K}^P\smash\downarrow^P_{UL_n^+}\rangle_{UL_n^+}
\ge m_i$. Note that the characters
$\chi_i\smash\uparrow_{UL_n^+}^P$ are irreducible by Clifford theory.
Hence, ${}^+\psi_\vartheta = (\vartheta \cdot
\hat{\lambda}^+)\smash\uparrow_{UL_n^+}^P$ is a  
subcharacter of $({}^s\nu)\smash\uparrow_{RQ_K}^P$. Computing the
degrees we get
\begin{eqnarray*}
{}^+\psi_\vartheta(1) & = & \frac12 q^{m-1}(q^{m-1}+1)(q-1)\vartheta(1)\\
& = & \frac12 q^{m-1}(q^{m-1}+1)(q-1) \cdot
\frac{|L_n^+|}{|P_{n-3}^+|} \cdot \vartheta_0(1)\\
& = & \frac12 q^{m-1}(q^{m-1}+1)(q-1) \cdot
\frac{(q^{m-1}-1) \cdot (q^{m-2}+1)}{q-1} \cdot \vartheta_0(1)
\end{eqnarray*}
and
\begin{eqnarray*}
({}^s\nu)\smash\uparrow_{RQ_K}^P(1) & = & \frac{|P|}{|RQ_K|} \cdot
  \nu(1)\\
& = & \frac{|U| \cdot |L_n|}{|R| \cdot (q-1) \cdot |P_{n-2}|}
  \cdot \frac12 q^{m-2} (q^{m-2}+1)(q-1) \cdot \vartheta_0(1)\\
& = & \frac12 q^{m-1}(q^{m-1}+1)(q-1) \cdot
\frac{(q^{m-1}-1) \cdot (q^{m-2}+1)}{q-1} \cdot \vartheta_0(1).
\end{eqnarray*}
Hence ${}^+\psi_\vartheta(1) = ({}^s\nu)\smash\uparrow_{RQ_K}^P(1)$
and then $({}^s\nu)\smash\uparrow_{RQ_K}^P = {}^+\psi_\vartheta$.

\medskip

(e)~Since $P_3$ does not have any characters of Type~$-$ we can assume 
$m \ge 3$. We denote the inflation of $\vartheta_0 \cdot
\hat{\lambda}_{n-2}^- \in \Irr(U_{n-2}L_{n-2}^-)$ to the group 
$R(A \times U_{n-2}L_{n-2}^-)$ also by $\vartheta_0 \cdot
\hat{\lambda}_{n-2}^- $. Hence 
\[
{}^s\vartheta_0 \cdot {}^s\hat{\lambda}_{n-2}^- \in \Irr({}^sR(A_{n-2} \times {}^s(U_{n-2}L_{n-2}^-))).
\]
An elementary matrix calculation shows that
$RP_{n-3}^- \le {}^sR(A_{n-2} \times {}^s(U_{n-2}L_{n-2}^-))$.

The proof now proceeds in several steps.

\medskip

\textbf{Step 1:} We have 
$({}^s\vartheta_0 \cdot
{}^s\hat{\lambda}_{n-2}^-)\smash\downarrow^{{}^sR(A_{n-2}  
\times {}^s(U_{n-2}L_{n-2}^-))}_{RP_{n-3}^-} = 
(\tilde{\vartheta}_0 \cdot \hat{\lambda}^-)\smash\downarrow^{UP_{n-3}^-}_{RP_{n-3}^-}$.

\textbf{Proof of Step 1:} Let $\mathbf{v} = \begin{bmatrix}v_{m-1}, 
  \dots, v_0, \dots, v_{m-2}', 0\end{bmatrix}^\tr \in \F_q^{n-2}$. Then
\begin{eqnarray*}
&&({}^s\vartheta_0 \cdot
{}^s\hat{\lambda}_{n-2}^-)(\mbu_n(\mathbf{v}))
= (\vartheta_0 \cdot \hat{\lambda}_{n-2}^-)(\mbu_n(\mathbf{v})^s)\\
& = & \vartheta_0(1) \xi((\mathbf{b}_{n-4}\begin{bmatrix}v_{m-2},
  \dots, v_0, \dots, v_{m-2}'\end{bmatrix}^\tr)_0)\\
& = & \vartheta_0(1) \xi((\mathbf{b}_{n-2}\mathbf{v})_0)
= \vartheta_0(1) \cdot \hat{\lambda}^-(\mbu_n(\mathbf{v})).
\end{eqnarray*}
It follows that $({}^s\vartheta_0 \cdot
{}^s\hat{\lambda}_{n-2}^-)\smash\downarrow^{{}^sR(A_{n-2}  
\times {}^s(U_{n-2}L_{n-2}^-))}_{RP_{n-3}^-} = \varphi \cdot 
(\hat{\lambda}^-\smash\downarrow^{UP_{n-3}^-}_{RP_{n-3}^-})$ where
$\varphi$ is the inflation of some character of $P_{n-3}^-$ to
$RP_{n-3}^-$. Note that $\varphi$ is uniquely determined by its
restriction to $P_{n-3}^-$. We have
\begin{eqnarray*}
&& \varphi(\mathbf{s}_n(
\mathbf{b}_{n-2}^{-1} \hspace{-0,1cm} \begin{bmatrix}
b & * & . & * & \hspace{- ,1cm} *\\
  & A & . & B & \hspace{- ,1cm} *\\
  & . & a & . & \hspace{- ,1cm} .\\
  & C & . & D & \hspace{- ,1cm} *\\
  &   &   &   & \hspace{- ,1cm} b^{-1}
\end{bmatrix} \hspace{- ,1cm} \mathbf{b}_{n-2}, a))\\
& = & ({}^s\vartheta_  \cdot
{}^s\hat{\lambda}_{n-2}^-)(
\mathbf{s}_n(
\mathbf{b}_{n-2}^{-1} \hspace{- ,1cm} \begin{bmatrix}
b & * & . & * & \hspace{- ,1cm} *\\
  & A & . & B & \hspace{- ,1cm} *\\
  & . & a & . & \hspace{- ,1cm} .\\
  & C & . & D & \hspace{- ,1cm} *\\
  &   &   &   & \hspace{- ,1cm} b^{-1}
\end{bmatrix} \hspace{- ,1cm} \mathbf{b}_{n-2}, a))\\
& = & (\vartheta_  \cdot \hat{\lambda}_{n-2}^-)(
\mathbf{s}_n(
\mathbf{b}_{n-2}^{-1} \hspace{- ,1cm} \begin{bmatrix}
b & * & . & * & \hspace{- ,1cm} *\\
  & A & . & B & \hspace{- ,1cm} *\\
  & . & a & . & \hspace{- ,1cm} .\\
  & C & . & D & \hspace{- ,1cm} *\\
  &   &   &   & \hspace{- ,1cm} b^{-1}
\end{bmatrix} \hspace{- ,1cm} \mathbf{b}_{n-2}, a)^s)
\end{eqnarray*}
\begin{eqnarray*}
& = & (\vartheta_  \cdot \hat{\lambda}_{n-2}^-)(
\mathbf{s}_n(
\mathbf{b}_{n-2}^{-1} \hspace{- ,1cm} \begin{bmatrix}
a &   &  &   & \hspace{- ,1cm}  \\
  & A & . & B & \hspace{- ,1cm}  \\
  & . & a & . & \hspace{- ,1cm}  \\
  & C & . & D & \hspace{- ,1cm}  \\
  &   &   &   & \hspace{- ,1cm} a
\end{bmatrix} \hspace{- ,1cm} \mathbf{b}_{n-2}, b) \cdot
\begin{bmatrix}1 & * & *\\   & I_{n-2} & *\\   &   & 1\end{bmatrix})\\
& = & \vartheta (\mathbf{s}_{n-2}(
\mathbf{b}_{n-4}^{-1} \hspace{- ,1cm} 
\begin{bmatrix}A & . & B\\. & a & .\\C & . & D
\end{bmatrix} \hspace{- ,1cm} \mathbf{b}_{n-4}, a))
= 
\tilde{\vartheta} (
\mathbf{s}_n(
\mathbf{b}_{n-2}^{-1} \hspace{- ,1cm} \begin{bmatrix}
b & * & . & * & \hspace{- ,1cm} *\\
  & A & . & B & \hspace{- ,1cm} *\\
  & . & a & . & \hspace{- ,1cm} .\\
  & C & . & D & \hspace{- ,1cm} *\\
  &   &   &   & \hspace{- ,1cm} b^{-1}
\end{bmatrix} \hspace{- ,1cm} \mathbf{b}_{n-2}, a)),
\end{eqnarray*}
thus $\varphi\smash\downarrow^{RP_{n-3}^-}_{P_{n-3}^-} =
\tilde{\vartheta}_0\smash\downarrow^{RP_{n-3}^-}_{P_{n-3}^-}$ and then 
\[
({}^s\vartheta_0 \cdot
{}^s\hat{\lambda}_{n-2}^-)\smash\downarrow^{{}^sR(A_{n-2}  
\times {}^s(U_{n-2}L_{n-2}^-))}_{RP_{n-3}^-} = \tilde{\vartheta}_0 \cdot 
(\hat{\lambda}^-\smash\downarrow^{UP_{n-3}^-}_{RP_{n-3}^-}) = 
(\tilde{\vartheta}_0 \cdot \hat{\lambda}^-)\smash\downarrow^{UP_{n-3}^-}_{RP_{n-3}^-}.
\]

\medskip

\textbf{Step 2:} The character $\vartheta \cdot \hat{\lambda}^-$ is a subcharacter of 
${}^s\nu\smash\uparrow_{RQ_K}^P\smash\downarrow^P_{UL_n^-}$.

\textbf{Proof of Step 2:} It follows from Step~1 that 
$\tilde{\vartheta}_0 \cdot \hat{\lambda}^-$ is an irreducible
constituent of $({}^s\vartheta_0 \cdot
{}^s\hat{\lambda}_{n-2}^-)\smash\downarrow^{{}^sR(A_{n-2}  
\times {}^s(U_{n-2}L_{n-2}^-))}_{RP_{n-3}^-}\smash\uparrow_{RP_{n-3}^-}^{UP_{n-3}^-}$. 
Thus 
\[
(\tilde{\vartheta}_0 \cdot \hat{\lambda}^-)\smash\uparrow_{UP_{n-3}^-}^{UL_n^-}
= \tilde{\vartheta}_0\smash\uparrow_{UP_{n-3}^-}^{UL_n^-} \cdot \hat{\lambda}^-
= \vartheta \cdot \hat{\lambda}^-
\]
is a subcharacter of $({}^s\vartheta_0 \cdot
{}^s\hat{\lambda}_{n-2}^-)\smash\downarrow^{{}^sR(A_{n-2}  
\times {}^s(U_{n-2}L_{n-2}^-))}_{RP_{n-3}^-}\smash\uparrow_{RP_{n-3}^-}^{UL_n^-}$. 
By construction, the character $\nu \in \Irr(RQ_K)$ is given by
$\nu = (\vartheta_0 \cdot \hat{\lambda}_{n-2}^-)\smash\uparrow_{R(A \times
  U_{n-2}L_{n-2}^-)}^{RQ_K}$ and hence 
\begin{equation} \label{eq:constnus-}
{}^s\nu = ({}^s\vartheta_0 \cdot {}^s\hat{\lambda}_{n-2}^-)\smash\uparrow_{{}^sR({}^sA \times
  {}^s(U_{n-2}L_{n-2}^-))}^{RQ_K} = ({}^s\vartheta_0 \cdot
{}^s\hat{\lambda}_{n-2}^-)\smash\uparrow_{{}^sR(A_{n-2} \times {}^s(U_{n-2}L_{n-2}^-))}^{RQ_K}.
\end{equation}
It follows from (\ref{eq:constnus-}) that ${}^s\vartheta_0 \cdot {}^s\hat{\lambda}_{n-2}^-$
is a constituent of ${}^s\nu\smash\downarrow^{RQ_K}_{{}^sR(A_{n-2} \times {}^s(U_{n-2}L_{n-2}^-))}$.
Hence $({}^s\vartheta_0 \cdot
{}^s\hat{\lambda}_{n-2}^-)\smash\downarrow^{{}^sR(A_{n-2} \times
  {}^s(U_{n-2}L_{n-2}^-))}_{RP_{n-3}^-}\smash\uparrow_{RP_{n-3}^-}^{UL_n^-}$ is a subcharacter of
${}^s\nu\smash\downarrow^{RQ_K}_{RP_{n-3}^-}\smash\uparrow_{RP_{n-3}^-}^{UL_n^-}$
and from the beginning of the proof of Step~2 we see that 
$\vartheta \cdot \hat{\lambda}^-$ is a subcharacter of
${}^s\nu\smash\downarrow^{RQ_K}_{RP_{n-3}^-}\smash\uparrow_{RP_{n-3}^-}^{UL_n^-}$. 
An elementary matrix calculation shows that
\[
RQ_K \cap UL_n^- = RP_{n-3}^-.
\]
Thus Mackey's Theorem implies that the character
${}^s\nu\smash\downarrow^{RQ_K}_{RP_{n-3}^-}\smash\uparrow_{RP_{n-3}^-}^{UL_n^-}$
is a subcharacter of 
${}^s\nu\smash\uparrow_{RQ_K}^P\smash\downarrow^P_{UL_n^-}$ and we can
conclude that $\vartheta \cdot \hat{\lambda}^-$ is a subcharacter of 
${}^s\nu\smash\uparrow_{RQ_K}^P\smash\downarrow^P_{UL_n^-}$.

\medskip

\textbf{Step 3:} We have $({}^s\nu)\smash\uparrow_{RQ_K}^P = {}^-\psi_\vartheta$.

\textbf{Proof of Step 3:} Write 
$\vartheta \cdot \hat{\lambda}^- = \sum_i m_i\chi_i$
with $\chi_i \in \Irr(UL_n^-)$. By Step~2 we have
\[
{}^s\nu\smash\uparrow_{RQ_K}^P\smash\downarrow^P_{UL_n^-}
= \sum_i m_i\chi_i + (\dots)
\]
and thus
$\langle\chi_i\smash\uparrow_{UL_n^-}^P,
({}^s\nu)\smash\uparrow_{RQ_K}^P \rangle_P = \langle\chi_i,
({}^s\nu)\smash\uparrow_{RQ_K}^P\smash\downarrow^P_{UL_n^-}\rangle_{UL_n^-}
\ge m_i$. Note that the characters
$\chi_i\smash\uparrow_{UL_n^-}^P$ are irreducible by Clifford theory.
Hence, ${}^-\psi_\vartheta = (\vartheta \cdot
\hat{\lambda}^-)\smash\uparrow_{UL_n^-}^P$ is a subcharacter of
$({}^s\nu)\smash\uparrow_{RQ_K}^P$. Computing the degrees we get
\begin{eqnarray*}
{}^-\psi_\vartheta(1) & = & \frac12 q^{m-1}(q^{m-1}-1)(q-1)\vartheta(1)\\
& = & \frac12 q^{m-1}(q^{m-1}-1)(q-1) \cdot
\frac{|L_n^-|}{|P_{n-3}^-|} \cdot \vartheta_0(1)\\
& = & \frac12 q^{m-1}(q^{m-1}-1)(q-1) \cdot
\frac{(q^{m-1}+1) \cdot (q^{m-2}-1)}{q-1} \cdot \vartheta_0(1)
\end{eqnarray*}
and
\begin{eqnarray*}
({}^s\nu)\smash\uparrow_{RQ_K}^P(1) & = & \frac{|P|}{|RQ_K|} \cdot
  \nu(1)\\
& = & \frac{|U| \cdot |L_n|}{|R| \cdot (q-1) \cdot |P_{n-2}|}
  \cdot \frac12 q^{m-2} (q^{m-2}-1)(q-1) \cdot \vartheta_0(1)\\
& = & \frac12 q^{m-1}(q^{m-1}-1)(q-1) \cdot
\frac{(q^{m-1}+1) \cdot (q^{m-2}-1)}{q-1} \cdot \vartheta_0(1).
\end{eqnarray*}
Hence ${}^-\psi_\vartheta(1) = ({}^s\nu)\smash\uparrow_{RQ_K}^P(1)$
and then $({}^s\nu)\smash\uparrow_{RQ_K}^P = {}^-\psi_\vartheta$.
\end{proof}

\section{Restriction of the Steinberg character}
\label{sec:restrictionSt}

We use the setting and notation from the previous sections. In
particular, we fix an odd integer $n=2m+1$ and let 
$G = G_n = \SO_{2m+1}(q)$. The restriction of the Steinberg character
$\St_G$ to the parabolic subgroup $P = P_n$ was already investigated
by Schm\"olzer in \cite{Schmoelzer}. See also
\cite[Section~3]{AnHissSteinberg} for a comparison to the symplectic case.

\begin{proposition}[\cite{Schmoelzer}, Corollary~2.3.8] \label{prop:resSt_SO}
Let $n \ge 5$ and $L' := L_n' \cong \SO_{n-2}(q)$. For each 
$\sigma \in \Irr(L')$ we denote its trivial extension to $L$ also
by~$\sigma$. Then
\begin{eqnarray*}
\sigma\smash\uparrow_L^P & = & {^1}\psi_\sigma +
\sum_{\mu\in\Irr(P_{n-2})} \langle \sigma\smash\downarrow^{L'}_{P_{n-2}},
\mu \rangle_{P_{n-2}} {^0}\psi_\mu \\
&& + \sum_{\vartheta\in\Irr(L_n^+)} \langle \sigma\smash\downarrow^{L}_{L_n^+},
\vartheta \rangle_{L_n^+} {^+}\psi_\vartheta
  + \sum_{\vartheta\in\Irr(L_n^-)} \langle \sigma\smash\downarrow^{L}_{L_n^-},
\vartheta \rangle_{L_n^-} {^-}\psi_\vartheta.
\end{eqnarray*}
\end{proposition}
\begin{proof}
Analogous to \cite[Proposition 3.2]{AnHissSteinberg}. 
\end{proof}
We are interested in the decomposition of $\St_G\smash\downarrow^G_P$
into irreducible characters. The following corollary reduces this
problem to calculations with characters of subgroups of $L$. In
Sections~\ref{sec:resuniSO5} and \ref{sec:resuniSO7}, this reduction
will be used to get a complete description of the restriction of the
Steinberg character for $n=5$ and to get a partial description for
$n=7$. 

\begin{corollary}[\cite{Schmoelzer}, Corollary~2.3.9] \label{cor:resSt_SO}
Suppose that $n \ge 5$. Then
\begin{eqnarray*}
\St_G\smash\downarrow^G_P & = & {^1}\psi_{\St_L} +
\sum_{\mu\in\Irr(P_{n-2})} \langle \St_{L'}\smash\downarrow^{L'}_{P_{n-2}},
\mu \rangle_{P_{n-2}} {^0}\psi_\mu \\
&& + \sum_{\vartheta\in\Irr(L_n^+)} \langle \St_L\smash\downarrow^{L}_{L_n^+},
\vartheta \rangle_{L_n^+} {^+}\psi_\vartheta
  + \sum_{\vartheta\in\Irr(L_n^-)} \langle \St_L\smash\downarrow^{L}_{L_n^-},
\vartheta \rangle_{L_n^-} {^-}\psi_\vartheta.
\end{eqnarray*}
\end{corollary}
\begin{proof}
This follows from~\cite[Proposition~6.3.3]{Carter} and
Proposition~\ref{prop:resSt_SO}. 
\end{proof}

\section{Character values on $U$}
\label{sec:charvalU}

As before, we fix an odd integer $n=2m+1$ and let $G = G_n = \SO_n(q)$. 
We determine the values of the irreducible characters of the 
parabolic subgroup $P=P_n$ on those conjugacy classes of $P$ which are
contained in $U = U_n$. In this section we always assume that $n \ge 5$.

Let $\{e_{m-1}, \dots, e_1,e_0,e_1', \dots, e_{m-1}'\}$ be the
standard basis of $\F_q^{n-2}$ as in Section~\ref{subsec:SO}. 
The conjugation action of $L$ on the normal abelian subgroup $U$ is
given by ${}^{\mbs_n(\mbx, a)} \mbu_n(\mbv) =
\mbu_n(a\mathbf{xv})$. Since $\mbx \in \SO_{n-2}(q)$ the values
$Q_{n-2}(\mbv)$ and $Q_{n-2}(a\mathbf{xv})$ only differ by a non-zero
factor which is a square in $\F_q^\times$. 

We define $\mbv_0, \mbv_1, \mbv_2 \in \F_q^{n-2}$ by $\mbv_0 := e_{m-1}$, 
$\mbv_1 := e_0$, $\mbv_2 := e_{m-1}+\nu''e_{m-1}'$, where 
$\nu'' \in \F_q^\times$ is a non-square if $q$ is odd and 
$\nu'' = 1$ if $q$ is even. We write $\mbz_j := \mbu_n(\mbv_j)$ 
($j=0,1,2$) for the corresponding elements of $U$. We know from 
Brauer's Permutation Lemma~\cite[Corollary~(6.33)]{Isaacs:94} 
and the results in Section~\ref{subsec:inertiasubgrps} that
the number of conjugacy classes of $P$ which are contained in $U$ is
four.

The set $\{1, \mbz_0, \mbz_1, \mbz_2\}$ is a full set of
representatives for the conjugacy classes of~$P$ which are contained
in $U$. For odd $q$, this already follows from $Q_{n-2}(\mbv_0) = 0$,
$Q_{n-2}(\mbv_1) = 1$ and $Q_{n-2}(\mbv_2) = \nu''$. For even $q$, 
the definition of $\SO_{n-2}(q)$ and $Q_{n-2}$ implies that 
$e_0$ is a common eigenvector of all $\mbx \in \SO_{n-2}(q)$ so that
$\mbz_1$ is not conjugate to $\mbz_2$ in $P$. Now, also in this case, 
the values $Q_{n-2}(\mbv_0) = 0$, $Q_{n-2}(\mbv_1) = 1$ and 
$Q_{n-2}(\mbv_2) = \nu''$ imply that $\{1, \mbz_0, \mbz_1, \mbz_2\}$ 
is a full set of representatives for the conjugacy classes of~$P$
which are contained in $U$. By a straightforward calculation, similar
to those in~Section~\ref{subsec:inertiasubgrps}, one can compute
the centralizers $C_P(\mbz_0)$, $C_P(\mbz_1)$ and their orders
explicitly. It follows that 
\begin{eqnarray}
|C_P(\mbz_0)|  & = & q^{m^2}(q-1)(q^{2m-4}-1)(q^{2m-6}-1) \cdots (q^2-1), \nonumber\\
|C_P(\mbz_1)| & =
& \begin{cases}2q^{m^2-m+1}(q^{m-1}-1)(q^{2m-4}-1) \cdots (q^2-1) & \text{for odd} \,\, q ,\\
q^{m^2}(q^{2m-2}-1)(q^{2m-4}-1) \cdots (q^2-1) & \text{for even} \,\, q ,
\end{cases} \label{eq:cenorduniP}\\
|C_P(\mbz_2)| & =
& \begin{cases}2q^{m^2-m+1}(q^{m-1}+1)(q^{2m-4}-1) \cdots (q^2-1) & \text{for odd} \,\, q ,\\
q^{m^2}(q^{2m-4}-1)(q^{2m-6}-1) \cdots (q^2-1) & \text{for even} \,\, q .\nonumber
\end{cases}
\end{eqnarray}
For $j=0,1,2$ let $c_j$ be the $P$-conjugacy class containing $\mbz_j$
and $C_j$ the $G$-conjugacy class containing $\mbz_j$. To express the
dependency on $n$ we also write $c_j^{(n)}$ and $C_j^{(n)}$, 
respectively. 

\begin{remark} \label{rmk:conjPG}
For $n \ge 5$ the $G$-conjugacy classes $C_0$, $C_1$, $C_2$ are pairwise distinct.
\end{remark}
\begin{proof}
It follows from Lemma~\ref{la:intersec_par} (a), that a full set of
$P$-$P$-double coset representatives in $G$ is given by $\{1,s,t\}$,
where $s$ and $t$ are the elements defined in~\eqref{eq:defst}.
Elementary matrix calculations give that
\begin{eqnarray*}
{^s}U\cap U & = & \{\mbu_n(\begin{bmatrix}v_{m-1} & 0 & \dots &
  0 \end{bmatrix}^\tr) \mid v_{m-1} \in \F_q\} \ \text{and}\\  
{^t}U\cap U & = & \{1\}.
\end{eqnarray*} 
Hence ${}^sc_j \cap U = \emptyset$ for $j=1,2$ and 
${}^s c_0 \cap U \subseteq c_0$ and ${}^tc_j \cap U = \emptyset$ for
all $j$, so the claim follows.
\end{proof}

\begin{lemma} \label{la:vals_psi_U}
 Let $n \ge 5$. Let $\mu$ be a character of $P_{n-2}$ and $\vartheta$
 a character of $L_n^+$ or~$L_n^-$. The values the irreducible
 characters of $P$ take on the unipotent elements $\mbz_0$, $\mbz_1$,
 $\mbz_2$ are given in Tables~\ref{tab:valsonunipot}
 and~\ref{tab:valsonunipoteven}.  
 \begin{table}
 \[ \begin{array}{|c|ccc|}
  \hline
  & \mbz_0 & \mbz_1 & \mbz_2\\
  \hline
  {^0}\psi_\mu 
  & -\mu(1) & (q^{m-1}-1)\mu(1) &
  -(q^{m-1}+1)\mu(1)\\
  {^+}\psi_\vartheta 
  & \frac{1}{2}q^{m-1}(q-1)\vartheta(1) & -q^{m-1}\vartheta(1) & 0 \\
  {^-}\psi_\vartheta 
  & -\frac{1}{2}q^{m-1}(q-1)\vartheta(1)& 0& q^{m-1}\vartheta(1)\\
  \hline
 \end{array}
\] 
\caption{Character values on $U$ for odd $q$.}
\label{tab:valsonunipot}
\end{table}
 \begin{table}
 \[ \begin{array}{|c|ccc|}
  \hline
  & \mbz_0 & \mbz_1 & \mbz_2\\
  \hline
  {^0}\psi_\mu 
  & -\mu(1) & (q^{2m-2}-1)\mu(1) & -\mu(1)\\
  {^+}\psi_\vartheta 
  & \frac{1}{2}q^{m-1}(q-1)\vartheta(1) &
  -\frac{1}{2}q^{m-1}(q^{m-1}+1)\vartheta(1) &
  -\frac{1}{2}q^{m-1}\vartheta(1) \\
  {^-}\psi_\vartheta 
  & -\frac{1}{2}q^{m-1}(q-1)\vartheta(1)&
  -\frac{1}{2}q^{m-1}(q^{m-1}-1)\vartheta(1)& \frac{1}{2}q^{m-1}\vartheta(1)\\
  \hline
 \end{array}
\] 
\caption{Character values on $U$ for even $q$.}
\label{tab:valsonunipoteven}
\end{table}
\end{lemma}

\begin{proof}
We use the notation from Sections~\ref{subsec:inertiasubgrps} and
\ref{subsec:charsPn}, except for the fact that in this proof we denote
irreducible characters of $P_{n-2}$ also by $\vartheta$ (and not by
$\mu$) to treat all characters in a uniform way.
Let $\varepsilon \in \{ 0,+,-\}$ and
$j \in \{0,1,2\}$ and let $\vartheta \in \Irr(P_{n-2})$ if
$\varepsilon = 0$ and $\vartheta \in \Irr(L_n^\pm)$ if 
$\varepsilon = \pm$. Since $U \unlhd P$ we get from the definition of
induced characters and the construction of ${^\varepsilon}\psi_\vartheta$ 
that 
\begin{equation} \label{eq:orbitsums}
{^\varepsilon}\psi_\vartheta(\mbz_j) =
\frac{1}{|I^\varepsilon|} \sum_{x \in P} \vartheta(1)
\lambda^\varepsilon(x \mbz_j x^{-1}) =
\frac{|C_P(\mbz_j)| \cdot \vartheta(1)}{|I^\varepsilon|} \sum_{\mbz \in c_j} 
\lambda^\varepsilon(\mbz).
\end{equation}
Since $|C_P(\mbz_j)|$ and $|I^\varepsilon|$ are known, we
only need to determine the sum of character values on the right hand
side of \eqref{eq:orbitsums}. 

The subgroup $A \le L$ defined in Section~\ref{subsec:Pn} acts
regularly on each of the conjugacy classes $c_j$ by 
${}^{\mathbf{s}_n(I_{n-2}, a)}\mbu_n(\mbw) 
= \mathbf{s}_n(I_{n-2}, a) \mbu_n(\mbw) \mathbf{s}_n(I_{n-2}, a)^{-1}
= \mbu_n(a\mbw)$. For any vector $\mbw = \begin{bmatrix}w_{m-1},
\dots, w_1,w_0,w_1', \dots, w_{m-1}'\end{bmatrix}^\tr \in \F_q^{n-2}$ we
have 
\begin{eqnarray*}
\sum_{a \in \F_q^\times}\lambda^0(\mbu_n(a\mbw)) & = & 
\sum_{a \in \F_q^\times}\xi(aw_{m-1}') = \begin{cases} q-1 &
  \text{if} \,\, w_{m-1}'=0,\\ -1 & \text{if} \,\,
  w_{m-1}' \neq 0, \end{cases} \quad \text{and}\\
\sum_{a \in \F_q^\times}\lambda^+(\mbu_n(a\mbw)) & = & 
\sum_{a \in \F_q^\times}\xi(aw_0) \hspace{0.5cm} = \begin{cases} q-1 &
  \text{if} \,\, w_0=0,\\ -1 & \text{if} \,\,
  w_0 \neq 0. \end{cases}
\end{eqnarray*}
We set ${}^0N_j := |\{\mbu_n(\mbw) \in c_j \mid w_{m-1}'=0\}|$ and
${}^+N_j := |\{\mbu_n(\mbw) \in c_j \mid w_0=0\}|$. By partitioning
$c_j$ into regular $A$-suborbits, we have
\begin{equation} \label{eq:orbitsum0}
\sum_{\mbz \in c_j} 
\lambda^\varepsilon(\mbz) = \frac{q-1}{q-1} \cdot {}^\varepsilon N_j +
\frac{|c_j|-{}^\varepsilon N_j}{q-1} \cdot (-1) = 
\frac{1}{q-1}(q \cdot {}^\varepsilon N_j - |c_j|).
\end{equation}
Suppose that $\varepsilon = 0$. For $j=0$ we have 
\begin{eqnarray*}
{}^0 N_0 & = & |\{\mbw \in \F_q^{n-2} \setminus \{0\} \mid Q_{n-2}(\mbw) =
0 \,\, \text{and} \,\, w_{m-1}'= 0\}|\\
& = & q \cdot |c_0^{(n-2)}| + q-1 = q \cdot (q^{2m-4}-1) + q-1 = q^{2m-3}-1,
\end{eqnarray*}
so that \eqref{eq:orbitsum0} evaluates to $-1$. Next, let $j=1$. If
$q$ is odd, we have
\begin{eqnarray*}
{}^0 N_1 & = & |\{\mbw \in \F_q^{n-2} \mid Q_{n-2}(\mbw) 
\,\, \text{is a square in} \,\, \F_q^\times \,\, \text{and} \,\, w_{m-1}'= 0\}|\\
& = & q \cdot |c_1^{(n-2)}| = \frac{q-1}{2} q^{m-1} (q^{m-2}+1),
\end{eqnarray*}
giving $\frac{1}{2}q^{m-1}(q-1)$ as the right hand side of
\eqref{eq:orbitsum0}. If $q$ is even then ${}^0 N_1 = q-1$, 
and the right hand side of \eqref{eq:orbitsum0} evaluates to $q-1$.
This gives the values of ${}^0\psi_\vartheta$ on~$\mbz_0$ and
$\mbz_1$. The values of ${}^0\psi_\vartheta$ on $\mbz_2$ follow from
row orthogonality relations for the characters of $U$. 

Suppose that $\varepsilon = +$. By \cite[Lemma~6.10]{JacobsonBAI} for
odd $q$ and even $q$ we have ${}^{+}N_0 = q^{m-2}(q^{m-1}+q-1) - 1$
giving $q^{m-1}-1$ for \eqref{eq:orbitsum0}. Similarly, we obtain
${}^+N_1 = \frac{q-1}{2}q^{m-2}(q^{m-1}-1)$ for odd $q$, and hence
the right hand side of \eqref{eq:orbitsum0} evaluates to $-q^{m-1}$ in 
this case. However, for even $q$ we obtain ${}^+N_1=0$, as the only
vector~$\mbw$ which is a scalar multiple of $e_0$ with $w_0=0$ is the
zero vector. Therefore \eqref{eq:orbitsum0} evaluates to $-1$. The
values of ${}^+\psi_\vartheta$ on $\mbz_2$ can be determined again
with the help of row orthogonality relations.

Suppose that $\varepsilon = -$. By the column orthogonality
relations for $\Irr(U)$, the sum of all coefficients of $\mu(1)$ 
and $\vartheta(1)$ in each column of Table~\ref{tab:valsonunipot}
and Table~\ref{tab:valsonunipoteven} has to be~$-1$. This gives the
values of ${}^-\psi_\vartheta$ and completes the proof. 
\end{proof}

\begin{remark} \label{rmk:valdegs}
Let $n\ge5$. Each character $\chi$ of $P=P_n$ can be written uniquely
as $\chi = {}^1\chi + {}^0\chi + {}^+\chi + {}^-\chi$ where
${}^\varepsilon\chi$ is the sum of the constituents of $\chi$ of Type
$\varepsilon$ and $\varepsilon \in \{0, 1, +, -\}$. So each
${}^\varepsilon\chi$ is of the form ${}^\varepsilon\psi_{\vartheta^\varepsilon}$ for
some not necessarily irreducible character $\vartheta^\varepsilon$. We call
${}^\varepsilon\chi$ the Type $\varepsilon$ component of $\chi$. 
Lemma~\ref{la:vals_psi_U} allows us to reconstruct the degrees  
${}^\varepsilon\chi(1)$ or equivalently the degrees $\vartheta^\varepsilon(1)$ 
from the values of~$\chi$ on the conjugacy classes of $P$ which are
contained in $U$ as follows: For odd $q$ we set 
\[
M := \begin{bmatrix}1 & q^{2m-2}-1 &
  \frac{q^{m-1}}{2}(q^{m-1}+1)(q-1) & 
\frac{q^{m-1}}{2}(q^{m-1}-1)(q-1)\\1 & -1 & \frac{q^{m-1}}{2}(q-1) & 
-\frac{q^{m-1}}{2}(q-1)\\1 & q^{m-1}-1 & -q^{m-1} & 0\\1 & -(q^{m-1}+1) & 0 & 
q^{m-1}\end{bmatrix},
\]
and for even $q$ we define
\[
M := \begin{bmatrix}
1 & q^{2m-2}-1 & \frac{q^{m-1}}{2}(q^{m-1}+1)(q-1) &
\frac{q^{m-1}}{2}(q^{m-1}-1)(q-1)\\
1 & -1 & \frac{q^{m-1}}{2}(q-1) & -\frac{q^{m-1}}{2}(q-1) \\
1 & q^{2m-2}-1 & -\frac{q^{m-1}}{2}(q^{m-1}+1) & -\frac{q^{m-1}}{2}(q^{m-1}-1) \\
1 & -1 & -\frac{q^{m-1}}{2} & \frac{q^{m-1}}{2}
\end{bmatrix}. 
\]
Because $\det(M) = q^{4m-2} \neq 0$ for odd $q$ and 
$\det(M) = \frac{1}{2}q^{5m-3} \neq 0$ for even $q$ we get from
Lemma~\ref{la:vals_psi_U} for all odd $q$ and all even $q$: 
\begin{equation} \label{eq:vals_mat}
\begin{bmatrix}\vartheta^1(1)\\\vartheta^0(1)\\\vartheta^+(1)\\\vartheta^-(1)\\\end{bmatrix}
  = M^{-1} \begin{bmatrix}\chi(1)\\\chi(\mbz_0)\\\chi(\mbz_1)\\\chi(\mbz_2)\end{bmatrix}.
\end{equation}
\end{remark}

\section{Restrictions of unipotent characters of $\SO_5(q)$}
\label{sec:resuniSO5}

When we speak of a unipotent character we always mean an irreducible
character. Recall (see \cite[Section~13.8]{Carter}) that for $n=2m+1$
the unipotent characters of the group $G = \SO_n(q)$ are parameterized
by symbols 
\[ 
\Lambda = \begin{pmatrix} \lambda_1 \: \lambda_2 \: \lambda_3 \: \cdots \: \lambda_r\\
\mu_1 \: \mu_2 \: \cdots \: \mu_s \end{pmatrix},
\]
for which $0\le \lambda_1 < \ldots < \lambda_r$ and 
$0 \le \mu_1 < \ldots < \mu_s$ are strictly increasing sequences of
non-negative integers and the second sequence may be
empty. Furthermore, the difference $r-s$ is odd and 
$\sum \lambda_i + \sum \mu_j - \lfloor (\frac{r+s-1}{2})^2 \rfloor = m$. 
We call $r-s$ the defect of $\Lambda$ and $m$ its rank. 
There is a bijection (see \cite[Section~11.4]{Carter}) between the set
of symbols of rank $m$ and defect $d=2d'+1$ and the set of
bipartitions $(\alpha,\beta)$ such that $|\alpha|+|\beta| = m - (d'^2+d')$.
Via this bijection we identify each symbol $\Lambda$ with the triple
$[\alpha,\beta,d]$ consisting of its corresponding bipartition and its
defect. We write $\chi_\Lambda$ for the unipotent character of $G$
corresponding to the symbol $\Lambda$. 

In this section we determine the decomposition into irreducible
characters of the restrictions of the unipotent characters of
$G_5 = \SO_5(q)$ to the maximal parabolic subgroup $P_5$. In the whole 
section we assume that $m=2$ and $n=2m+1=5$. We set 
$G:=\SO_5(q)$, $P:=P_5$, $L:=L_5$, and $L^\pm := L_5^\pm$ for
brevity. The degrees of the unipotent characters of $\SO_5(q)$ and
their labels are given in Table~\ref{tab:labelsSO5}. 

\begin{table}[h]
 \centering
 \begin{tabular}{|c|c|c||c|c|c|}
  \hline
  Bipartition & Symbol & Degree &  Bipartition & Symbol & Degree \\
  \hline
  ${[2,-,1]}$ & $\begin{pmatrix}2\\-\end{pmatrix}$ & $1$ &
    ${[1,1,1]}$ & $\begin{pmatrix}0 \: 2\\1\end{pmatrix}$ &
     $\frac{1}{2}q(q+1)^2$\\
    ${[-,-,3]}$ & $\begin{pmatrix} 0 \: 1\: 2\\ - \end{pmatrix}$ &
     $\frac{1}{2}q(q-1)^2$ &
   ${[-,2,1]}$ & $\begin{pmatrix} 0\: 1\\2 \end{pmatrix}$ &
    $\frac{1}{2}q(q^2+1)$\\
   ${[1^2,-,1]}$ & $\begin{pmatrix}1\: 2\\ 0\end{pmatrix}$ & 
    $\frac{1}{2}q(q^2+1)$ &
  ${[-,1^2,1]}$ & $\begin{pmatrix} 0\: 1\: 2\\ 1\:2 \end{pmatrix}$ & $q^4$\\
  \hline
 \end{tabular}
 \caption{Labels and degrees of the unipotent characters of $\SO_5(q)$.}
 \label{tab:labelsSO5}
\end{table}

As in Section~\ref{sec:charvalU} we denote the $P$-conjugacy class
of the unipotent element~$\mbz_j$ by~$c_j$ for $j=0,1,2$. For odd
prime powers $q$, the conjugacy classes and the values of the unipotent characters
of $G=\SO_5(q)$ were computed by Frank L\"ubeck (private communication). 
The group $G$ has exactly five unipotent conjugacy classes and the
order of the centralizer in $G$ of representatives for these classes
are: $|G|$, $q^4(q^2-1)$, $2q^3(q-1)$, $2q^3(q+1)$ and~$q^2$,
respectively. Remark~\ref{rmk:conjPG} and the centralizer
orders~\eqref{eq:cenorduniP} determine the fusion of the classes
$c_0$, $c_1$, $c_2$ into the unipotent classes of $G$ uniquely so that
we can read off the values of the unipotent characters of $G$ on the
elements $\mbz_0$, $\mbz_1$ and $\mbz_2$ in Table~\ref{tab:uniGzj}
for odd $q$ from L\"ubeck's data. In the table zeros are replaced by dots.

To determine the values of the unipotent characters of $G$ on $\mbz_j$
for even $q$ we identify $G=\SO_5(q)$ with $\Sp_4(q)$ via the isomorphism
described in Remark~\ref{rmk:even_iso_Sp}. The conjugacy classes and
the irreducible characters of $\Sp_4(q)$ for even $q$ were determined by
Enomoto~\cite{Enomoto}. Using the notation for the unipotent conjugacy
classes and the representatives in~\cite{Enomoto} we see that 
$\mbz_0 \in A_{31}$,  $\mbz_1 \in A_2$ and $\mbz_2 \in A_{32}$.
Now the character table of $\Sp_4(q)$ in~\cite{Enomoto} or
CHEVIE~\cite{CHEVIE} gives the values in Table~\ref{tab:uniGzj} for
even $q$. 

\begin{table}[h]
 \centering
 \[ 
 \begin{array}{|c|ccc|ccc|}
 \hline 
 & \multicolumn{3}{|c|}{q \,\, \text{odd}} & \multicolumn{3}{|c|}{q
     \,\, \text{even}}\\
 \Lambda & \mbz_0 & \mbz_1 & \mbz_2 & \mbz_0 & \mbz_1 & \mbz_2 \\\hline
 {[-,-,3]} & -\frac{1}{2}q(q-1) & . & q & -\frac12 q (q-1) & -\frac12
 q (q-1) & \frac12 q\\
 {[1^2,-,1]} & -\frac{1}{2}q(q-1) & q & . & -\frac12 q (q-1) & \frac12 q
 (q+1) & \frac12 q\\
 {[1,1,1]} & \frac{1}{2}q(q+1) & q & . & \frac12 q (q+1) & \frac12 q
 (q+1) & \frac12 q\\
 {[-,2,1]} & \frac{1}{2}q(q+1) & . & q & \frac12 q(q+1) & -\frac12 q
 (q-1) & \frac12 q\\
 \hline
 \end{array}
\]
 \caption{Values $\chi_\Lambda(\mbz_j)$ of unipotent characters of
   $\SO_5(q)$.}
 \label{tab:uniGzj}
\end{table}

Additionally, we have the values $\chi_{[2,-,1]}(\mbz_j) = 1$ for all
$j$ for the trivial character and $\chi_{[-,1^2,1]}(\mbz_j) = 0$ for
all $j$ for the Steinberg character. 
To describe the restriction of the irreducible characters of 
$G = \SO_5(q)$ to the parabolic subgroup~$P$ we collect some
information on the conjugacy classes and irreducible characters of
some subgroups of $G$.

\begin{remark} \label{rmk:charsubSO5}
(a)~The parabolic subgroup $P_3 \cong q:(q-1)$ of 
$L'  := L_5' \cong \SO_3(q)$ is a Borel subgroup and has $q-1$ linear
characters and a unique non-linear irreducible character, which we
denote by $\mu$. Its degree is $\mu(1) = q-1$. 

\smallskip

\noindent (b)~By \cite[Theorem 11.4]{Taylor} the group 
$L^\pm \cong \GO_2^\pm(q)$ is dihedral of order $2(q \mp 1)$. More
specifically: $L^\pm$ has a cyclic normal subgroup $K^\pm$ of index
$2$, and an outer involution acts on $K^\pm$ by inverting each
element. For odd $q$ we have $K^\pm = L^\pm \cap L'$. In other words:
$K^\pm = \SO_2^\pm(q)$ when we identify $L^\pm$ with $\GO_2^\pm(q)$ and
$q$ is odd.

For odd $q$, the group~$L^\pm$ has four linear characters
$1_{L^\pm}$, $\nu_1^\pm$, $\nu_2^\pm$, $\nu_3^\pm$ and 
$\frac{q\mp 1}{2} - 1$ irreducible characters $\chi_j^\pm$ of
degree~$2$ and we choose the numbering so that 
$K^\pm \le \ker(\nu_1^\pm)$~and 
$K^\pm \not\le \ker(\nu_2^\pm), \ker(\nu_3^\pm)$. We will see in
Theorem~\ref{thm:unipotresSO5} that exactly one of the characters
$\nu_2^\pm$, $\nu_3^\pm$ is a constituent of
$\St_L\smash\downarrow^L_{L^\pm}$ and we choose the notation so 
that $\nu_3^\pm$ is this constituent. The group $L^\pm$ has exactly
two conjugacy classes which are not contained in $K^\pm$. Both 
of them consist of involutions and have size $\frac{q \mp 1}{2}$. 

For even $q$, the group~$L^\pm$ has two linear characters
$1_{L^\pm}$, $\nu_1^\pm$ and $\frac{q-1\mp 1}{2}$ irreducible
characters $\chi_j^\pm$ of degree~$2$ and we have 
$K^\pm \le \ker(\nu_1^\pm)$. The group $L^\pm$ has only one
conjugacy class which is not contained in $K^\pm$. This conjugacy
class is the unique conjugacy class of involutions and has size $q \mp 1$. 

For odd $q$ and for even $q$ we write $\Xi^\pm$ for the sum of all
irreducible characters~$\chi_j^\pm$ of of $L^\pm$ degree $2$. 
\end{remark}

\begin{theorem}\label{thm:unipotresSO5}
For odd $q$ and even $q$ the unipotent characters of $\SO_5(q)$
restricted to $P_5$ decompose as given in Table~\ref{tab:unipotresSO5}.
\end{theorem}
\begin{table}[h]
 \[
 \begin{array}{|c||c|c|c|c|}
  \hline
   & \multicolumn{4}{|c|}{\chi_\Lambda\downarrow^G_P}\\
  \Lambda & \text{Type} \,\, 1 & \text{Type} \,\, 0 & \text{Type} \,\, + & \text{Type} \,\, - \\
  \hline
 {[2,-,1]}   & {[1,-,1]} &&& \\
 {[-,-,3]}   & & & & \nu_1^- \\
 {[1^2,-,1]} & {[1,-,1]} & 1_{P_3} & & 1_{L^-} \\
 {[1,1,1]}   & {[1,-,1]}+{[-,1,1]} & 1_{P_3} & 1_{L^+} & \\
 {[-,2,1]}   & {[-,1,1]} & & \nu_1^+ & \\
 {[-,1^2,1]} & {[-,1,1]} & 1_{P_3}+\mu & 1_{L^+}+\nu_1^+(+\nu_3^+)+\Xi^+ & (\nu_3^-+)\Xi^- \\\hline
 \end{array}
 \]
 \caption{Unipotent characters of $\SO_5(q)$ restricted to $P_5$.}
\label{tab:unipotresSO5}
\end{table}

\begin{remark} \label{rmk:rmkresSO5}
The first column of Table~\ref{tab:unipotresSO5} contains the symbols
parameterizing the unipotent characters $\chi_\Lambda$ of $G = \SO_5(q)$.
The second column lists symbols parameterizing characters $\sigma$
of $L$ such that ${}^1\psi_\sigma$ is the Type~$1$ component of
$\chi_\Lambda\smash\downarrow^G_P$. The entries in columns 3-5 are
characters $\vartheta$ such that ${}^\varepsilon\psi_\vartheta$ is the
Type~$\varepsilon$ component of $\chi_\Lambda\smash\downarrow^G_P$.

The characters $\mu$, $\nu_j^\pm$ and $\Xi^\pm$ are defined in
Remark~\ref{rmk:charsubSO5}. The characters in brackets in the last
row of Table~\ref{tab:unipotresSO5} only exist for odd $q$. More
precisely: For odd prime powers $q$, the Type~$+$ component of 
$\chi_{[-,1^2,1]}\smash\downarrow^G_P$ is 
${}^+\psi_{1_{L^+}+\nu_1^++\nu_3^++\Xi^+}$ and the Type~$-$
component is ${}^-\psi_{\nu_3^-+\Xi^-}$. For even prime powers $q$,
the Type~$+$ component of $\chi_{[-,1^2,1]}\smash\downarrow^G_P$ is 
${}^+\psi_{1_{L^+}+\nu_1^++\Xi^+}$ and the Type~$-$ component is
${}^-\psi_{\Xi^-}$. 
\end{remark}

\begin{proof} (of Theorem~\ref{thm:unipotresSO5})
Let $\chi_\Lambda$ be a unipotent character of $G$. As in
Remark~\ref{rmk:valdegs} we write the restriction of $\chi_\Lambda$ to
$P$ as $\chi_\Lambda\smash\downarrow^G_P = {}^1\chi + {}^0\chi +
{}^+\chi + {}^-\chi$ where 
${}^\varepsilon\chi = {}^\varepsilon\psi_{\vartheta^\varepsilon}$ is
the Type~$\varepsilon$ component of $\chi_\Lambda\smash\downarrow^G_P$. 
The degrees $\vartheta^\varepsilon(1)$ can be computed from
Table~\ref{tab:uniGzj} via~\eqref{eq:vals_mat}. The result is given in
Table~\ref{tab:resSO5degs}; it turns out that it does not depend on
whether $q$ is odd or even.

\begin{table}
 \[
 \begin{array}{|c|cccc|}
 \hline \Lambda & \vartheta^1(1) & \vartheta^0(1) & \vartheta^+(1) & \vartheta^-(1)\\\hline
 {[2,-,1]}   & 1 & . & . & . \\
 {[-,-,3]}   & . & . & . & 1 \\
 {[1^2,-,1]} & 1 & 1 & . & 1 \\
 {[1,1,1]}   & q+1 & 1 & 1 & .\\
 {[-,2,1]}   & q & . & 1 & . \\
 {[-,1^2,1]} & q & q & q & q\\\hline
 \end{array}
 \]
 \caption{Component degrees of the restrictions of the unipotent
   irreducible characters of $\SO_5(q)$ to $P_5$ for odd $q$ and 
   even $q$.} 
\label{tab:resSO5degs}
\end{table}

The first row of Table~\ref{tab:unipotresSO5} is trivial. Also, the
first column of Table~\ref{tab:unipotresSO5} can be determined from 
Table~\ref{tab:resSO5degs} and the combinatorics of bipartitions which
encode Harish-Chandra induction and restriction. In particular, we
obtain 
\[ 
\R{G}{L}(1_L) = \R{G}{L}(\chi_{[1,-,1]}) = \chi_{[2,-,1]} + \chi_{[1^2,-,1]} + \chi_{[1,1,1]}.
\]
By Proposition~\ref{prop:resSt_SO} we have 
$1_L\smash\uparrow_L^P = {^1}\psi_{1_L} + {^0}\psi_{1_{P_3}} +
{^+}\psi_{1_{L^+}} + {^-}\psi_{1_{L^-}}$. Theorem~\ref{thm:indnu}~(b)
and the paragraph preceding it now yield for $\sigma = 1_L$ and 
$\Sigma = 1_L + \St_L$ the equation 
\[ 
\R{G}{L}(1_L)\downa{G}{P} = 
   3\cdot{^1}\psi_{1_{L}} + {^1}\psi_{\St_L} + 2\cdot{^0}\psi_{1_{P_3}} + 
   {^+}\psi_{1_{L^+}} + {^-}\psi_{1_{L^-}}.
\]
Rows 3 and 4 of Table~\ref{tab:unipotresSO5} now follow
from Table~\ref{tab:resSO5degs}. Next, we determine the restriction of the
Steinberg character $\St_G = \chi_{[-,1^2,1]}$ to $P$. To apply
Corollary~\ref{cor:resSt_SO} we have to describe the decomposition 
of $\St_{L'}\smash\downarrow^{L'}_{P_3}$ and of
$\St_L\smash\downarrow^L_{L^\pm}$ into irreducibles. 

Since $\St_{L'}$ vanishes on the non-trivial unipotent elements, the
non-linear character $\mu \in \Irr(P_3)$ of degree $q-1$ is a
constituent of $\St_{L'}\smash\downarrow^{L'}_{P_3}$. Since
$1_{P_3}\smash\uparrow_{P_3}^{L'} = 1_{L'} + \St_{L'}$, Frobenius
reciprocity implies that 
\begin{equation} \label{eq:resStL'}
\St_{L'}\downa{L'}{P_3} = 1_{P_3} + \mu.
\end{equation}
Next, we consider the restriction of $\St_L$ to $L^\pm$ and $K^\pm$,
where $K^\pm$ is the cyclic normal subgroup of $L^\pm$ of index $2$
which is defined in Remark~\ref{rmk:charsubSO5}~(b). 
Suppose that $q$ is odd. By~\cite[Theorem~6.5.9]{Carter} we have
$\St_L(1) = q$ and $\St_L(\mbx) = 1$ for all $\mbx \in K^+ \setminus \{1\}$. Thus, 
$\St_L\smash\downarrow^L_{K^+}$ is the sum of the trivial character 
$1_{K^+}$ and the regular character of $K^+$. It follows that 
exactly one of the characters $\nu_2^+$, $\nu_3^+$ is a constituent 
of~$\St_L\smash\downarrow^L_{L^+}$ and that this constituent occurs
with multiplicity one. By definition, this constituent is $\nu_3^+$ 
so that we get 
\begin{eqnarray*}
\St_L\smash\downarrow^L_{L^+} & = & 1_{L^+} + \nu_1^+ + \nu_3^+ + \Xi^+ \quad \text{or}\\
\St_L\smash\downarrow^L_{L^+} & = & 2 \cdot 1_{L^+} + \nu_3^+ + \Xi^+ \quad \text{or}\\
\St_L\smash\downarrow^L_{L^+} & = & 2 \cdot \nu_1^+ + \nu_3^+ + \Xi^+.
\end{eqnarray*}
Let $\mbx$, $\mbx'$ be representatives for the two conjugacy classes
of $L^+$ which are not contained in~$K^+$. Hence, $\mbx$ and $\mbx'$
are involutions and by~\cite[Theorem~6.5.9]{Carter} we have
$\St_L(\mbx), \St_L(\mbx') \in \{\pm 1\}$. Thus
\[
\langle 1_{L^+}, \St_L\smash\downarrow^L_{L^+} \rangle_{L^+} = 1 +
\frac{\frac{q-1}{2}(\St_L(\mbx) + \St_L(\mbx'))}{2(q-1)} = 1 +
\frac14(\St_L(\mbx) + \St_L(\mbx')).
\]
It follows that 
$\langle 1_{L^+}, \St_L\downarrow^L_{L^+} \rangle_{L^+} = 1$ and
therefore $\St_L\smash\downarrow^L_{L^+} = 1_{L^+} + \nu_1^+ + \nu_3^+ + \Xi^+$. 
Next, we consider the restriction $\St_L\smash\downarrow^L_{L^-}$ for
odd $q$. By~\cite[Theorem~6.5.9]{Carter} we have $\St_L(1) = q$ and
$\St_L(\mbx) = -1$ for all $\mbx \in K^- \setminus \{1\}$. Thus, 
$\St_L\smash\downarrow^L_{K^-}$ is the regular character of $K^-$ with
the trivial character $1_{K^-}$ removed. It follows that exactly one
of $\nu_2^-$, $\nu_3^-$ is a constituent of $\St_L\smash\downarrow^L_{L^-}$ 
and that this constituent occurs with multiplicity one. By definition,
this constituent is $\nu_3^-$ and thus 
$\St_L\smash\downarrow^L_{L^-} = \nu_3^- + \Xi^-$. 

Now suppose that $q$ is even. As above, we see that 
$\St_L\smash\downarrow^L_{K^+}$ is the sum of $1_{K^+}$ and the
regular character of $K^+$ and that 
$\langle 1_{L^+}, \St_L\smash\downarrow^L_{L^+} \rangle_{L^+} = 1$. 
Hence, in this case we obtain that 
$\St_L\smash\downarrow^L_{L^+} = 1_{L^+} + \nu_1^+ + \Xi^+$. 
As above, we see that $\St_L\smash\downarrow^L_{K^-}$ is the regular
character of $K^-$ with the trivial character removed and get 
$\St_L\smash\downarrow^L_{L^-} = \Xi^-$. 

Applying Corollary~\ref{cor:resSt_SO} we get
\begin{equation} \label{eq:steinSO5odd}
\St_G\smash\downarrow^G_P = \chi_{[-,1^2,1]}\smash\downarrow^G_P = 
{}^1\psi_{\St_L} + {}^0\psi_{1_{P_3}+\mu} + 
{}^+\psi_{1_{L^+}+\nu_1^++\nu_3^++\Xi^+} + 
{}^-\psi_{\nu_3^- + \Xi^-}
\end{equation}
for odd $q$ and
\begin{equation} \label{eq:steinSO5even}
\St_G\smash\downarrow^G_P = \chi_{[-,1^2,1]}\smash\downarrow^G_P = 
{}^1\psi_{\St_L} + {}^0\psi_{1_{P_3}+\mu} + 
{}^+\psi_{1_{L^+}+\nu_1^++\Xi^+} + {}^-\psi_{\Xi^-}
\end{equation}
for even $q$, which proves the entries in row 6 of
Table~\ref{tab:unipotresSO5}. 
Next, we determine row~5. We have
\begin{equation} \label{eq:HCStSO5}
\R{G}{L}(\St_L) = \R{G}{L}(\chi_{[-,1,1]}) = \chi_{[1,1,1]} + \chi_{[-,1^2,1]} + \chi_{[-,2,1]}
\end{equation}
and we want to determine $\R{G}{L}(\St_L)\smash\downarrow^G_P$. Hence, we
have to compute the summands on the right hand side of
\eqref{eq:LGP} where $\sigma = \St_L$. Since conjugation with $t$
permutes the unipotent characters of $L$ 
we have ${}^t\sigma = {}^t\St_L= \St_L$. Hence, Theorem~\ref{thm:indnu}~(a)
and \cite[Proposition~6.3.3]{Carter} imply that   
\begin{equation} \label{eq:righttermSO5}
{}^t\sigma\smash\downarrow^{{}^tP}_{L}\smash\uparrow_{L}^P =
\St_L\smash\uparrow_L^P = \St_G\smash\downarrow^G_P.
\end{equation}
Suppressing symbols for inflation as in Section~\ref{subsec:resHC}
we get from \eqref{eq:resStL'} that 
\begin{equation} 
{}^s\sigma\smash\downarrow^{{}^sP}_{RQ_K}\smash\uparrow_{RQ_K}^P = 
1_{P_3}\smash\uparrow_{RQ_K}^P + \mu\smash\uparrow_{RQ_K}^P =
{}^1\psi_{1_{L_3}}\smash\uparrow_{RQ_K}^P + {^+}\psi_{1_{L_3^+}}\smash\uparrow_{RQ_K}^P.
\end{equation}
Note that $L_3^+ = \{1\}$. From Theorem~\ref{thm:indnu} (b) and (d)
we obtain
\begin{eqnarray} \label{eq:s_termSO5}
{}^1\psi_{1_{L_3}}\smash\uparrow_{RQ_K}^P & = & {}^0\psi_{1_{P_3}} +
{}^1\psi_{\R{L_5}{L_3}(1_{L_3})} = {}^0\psi_{1_{P_3}} +
{}^1\psi_{1_L} + {}^1\psi_{\St_L} ,\\
{^+}\psi_{1_{L_3^+}}\smash\uparrow_{RQ_K}^P & = &
{}^+\psi_{\R{L^+}{L_3^+}(1_{L_3^+})} = {}^+\psi_{1_{L^+}} +
{}^+\psi_{\nu_1^+}. \nonumber
\end{eqnarray}
Thus, we get from \eqref{eq:LGP} and
\eqref{eq:HCStSO5}-\eqref{eq:s_termSO5} that
\begin{eqnarray*}
\R{G}{L}(\St_L)\smash\downarrow^G_P & = & {}^1\psi_{\St_L} + {}^0\psi_{1_{P_3}} +
{}^1\psi_{1_L} + {}^1\psi_{\St_L} + {}^+\psi_{1_{L^+}} +
{}^+\psi_{\nu_1^+} + \St_G\smash\downarrow^G_P\\
& = & \chi_{[1,1,1]}\smash\downarrow^G_P +
\chi_{[-,2,1]}\smash\downarrow^G_P + \St_G\smash\downarrow^G_P.
\end{eqnarray*}
Since we already know $\chi_{[1,1,1]}\smash\downarrow^G_P$ we get
$\chi_{[-,2,1]}\smash\downarrow^G_P = {}^1\psi_{\St_L} + {}^+\psi_{\nu_1^+}$,
completing row~5 of Table~\ref{tab:unipotresSO5}.

Finally, we consider $\chi_{[-,-,3]}\smash\downarrow^G_P$. We know
from Table~\ref{tab:resSO5degs} that 
$\chi_{[-,-,3]}\downa{G}{P} = {^-}\psi_\vartheta$ for a linear
character $\vartheta \in \Irr(L^-)$. Therefore from the definition 
of ${^-}\psi_\vartheta$ it follows
\[ 
\chi_{[-,-,3]}\downa{G}{L^-} = {^-}\psi_\vartheta\downa{P}{L^-} =
\vartheta\upa{L}{L^-}\downa{L}{L^-}.
\]
Suppose that $q$ is odd. It follows from
Remark~\ref{rmk:charsubSO5}~(b) that $L^-$ has exactly three conjugacy
classes of involutions and we choose representatives $\tau_1$,
$\tau_2$, $\tau_3$ for these classes such that 
$\tau_1, \tau_2 \not\in K^-$ and $\tau_3 \in Z(L^-)$. Let $c^L_j$
be the $L$-conjugacy class of~$\tau_j$ for $j=1,2,3$. We have already
seen that $\St_L\smash\downarrow^L_{L^-} = \nu_3^- + \Xi^-$. Since
$\Xi^-$ vanishes on $L^- \setminus K^-$ we have $\St_L(\tau_1) =
\nu_3^-(\tau_1) \neq \nu_3^-(\tau_2) = \St_L(\tau_2)$, hence 
$c^L_1 \neq c^L_2$. Also from Remark~\ref{rmk:charsubSO5}~(b) we get
that $c^L_3 \subseteq L' \unlhd L$ and 
$c^L_1, c^L_2 \not\subseteq L'$. Thus $c^L_1$, $c^L_2$, $c^L_3$ are
pairwise distinct. Since the conjugacy classes of 
$L = A \times L' \cong \F_q^\times \times \SO_3(q)$ are known we see
that $|C_L(\tau_1)|, |C_L(\tau_2)| \in \{2(q^2-1), 2(q-1)^2\}$. Hence 
\[
\chi_{[-,-,3]}(\tau_j) = \vartheta\smash\uparrow_{L^-}^L(\tau_j) \in 
\left\{ \frac{q^2-1}{2} \vartheta(\tau_j), \frac{(q-1)^2}{2}
\vartheta(\tau_j) \right\}
\]
for $j=1,2$. From the character values of $\chi_{[-,-,3]}$ 
we get $\vartheta(\tau_1) = \vartheta(\tau_2) = -1$ and hence 
$\vartheta = \nu_1^-$.

Suppose that $q$ is even. There is just one conjugacy class of
involutions in~$L^-$; let $\tau$ be a representative for this
class. The group $L \cong \F_q^\times \times \SO_3(q) \cong
\F_q^\times \times \SL_2(q)$ has only one conjugacy class of
involutions and we see that $|C_L(\tau)| = q(q-1)$. Hence
we compute $\chi_{[-,-,3]}(\tau) = \vartheta\smash\uparrow_{L^-}^L(\tau) 
= \frac12 q(q-1) \vartheta(\tau)$. From the character table of 
$\SO_5(q) \cong \Sp_4(q)$ in CHEVIE we get
$\chi_{[-,-,3]}(\tau) = -\frac12 q(q-1)$, thus $\vartheta = \nu_1^-$.
This completes the proof of Theorem~\ref{thm:unipotresSO5}.
\end{proof}

\section{Restrictions of unipotent characters of $\SO_7(q)$}
\label{sec:resuniSO7}

In this section we obtain some information on the decomposition
into irreducible characters of the restrictions of the unipotent
characters of $G_7 = \SO_7(q)$ to the parabolic subgroup $P_7$. 
In the whole section we assume that $m=3$ and $n=2m+1=7$. We
set $G:=\SO_7(q)$, $P:=P_7$, $L:=L_7$, and $L^\pm := L_7^\pm$ for
brevity. The degrees of the unipotent characters of $\SO_7(q)$ and
their labels are given in Table~\ref{tab:labelsSO7}. We
use the abbreviations $\phi_1 := q-1$, $\phi_2 := q+1$, 
$\phi_3 := q^2+q+1$, $\phi_4 := q^2+1$ and $\phi_6 := q^2-q+1$.

\begin{table}
 \centering
 \begin{tabular}{|c|c|c||c|c|c|}
  \hline
  Bipartition & Symbol & Degree &  Bipartition & Symbol & Degree \\
  \hline
  ${[3,-,1]}$ & $\begin{pmatrix}3\\-\end{pmatrix}$ & $1$ &
    ${[1^2,1,1]}$ & $\begin{pmatrix}1\:2\\1\end{pmatrix}$ & 
     $q^3\phi_3\phi_6$\\
  ${[2,1,1]}$ & $\begin{pmatrix}0\:3\\1\end{pmatrix}$ &
      $\frac{1}{2}q\phi_3\phi_4$ & ${[1,1^2,1]}$ &
      $\begin{pmatrix}0\:1\:3\\1\:2\end{pmatrix}$ & 
      $\frac{1}{2}q^4\phi_3\phi_4$\\ 
  ${[-,3,1]}$ & $\begin{pmatrix}0\:1\\3\end{pmatrix}$ &
      $\frac{1}{2}q\phi_4\phi_6$ & ${[-,2\,1,1]}$ &
      $\begin{pmatrix}0\:1\:2\\1\:3\end{pmatrix}$ & 
      $\frac{1}{2}q^4\phi_2^2\phi_6$\\
  ${[2\,1,-,1]}$ & $\begin{pmatrix}1\:3\\0\end{pmatrix}$ &
      $\frac{1}{2}q\phi_2^2\phi_6$ & ${[1^3,-,1]}$ &
      $\begin{pmatrix}1\:2\:3\\0\:1\end{pmatrix}$ & 
      $\frac{1}{2}q^4\phi_4\phi_6$\\
  ${[1,-,3]}$ & $\begin{pmatrix}0\:1\:3\\-\end{pmatrix}$ &
      $\frac{1}{2}q\phi_1^2\phi_3$ & ${[-,1,3]}$ &
      $\begin{pmatrix}0\:1\:2\:3\\1\end{pmatrix}$ & 
      $\frac{1}{2}q^4\phi_1^2\phi_3$\\
  ${[1,2,1]}$ & $\begin{pmatrix}0\:2\\2\end{pmatrix}$ &
      $q^2\phi_3\phi_6$ & ${[-,1^3,1]}$ &
      $\begin{pmatrix}0\:1\:2\:3\\1\:2\:3\end{pmatrix}$ & $q^9$\\
  \hline
 \end{tabular}
 \caption{Labels and degrees of the unipotent characters of $\SO_7(q)$.}
 \label{tab:labelsSO7}
\end{table}

As in Section~\ref{sec:charvalU} we denote the $P$-conjugacy class
of the unipotent element~$\mbz_j$ by~$c_j$ for $j=0,1,2$. The
conjugacy classes and the values of the unipotent characters 
of $G=\SO_7(q)$ were computed by Frank L\"ubeck (private communication). 
For odd $q$ the group $G$ has $10$ unipotent conjugacy classes
and for even $q$ the group~$G$ has $12$ unipotent classes.
Remark~\ref{rmk:conjPG} and the centralizer
orders~\eqref{eq:cenorduniP} determine the fusion of the classes 
$c_0$, $c_1$, $c_2$ into the unipotent classes of $G$ uniquely so that
we can read off the values of the unipotent characters of $G$ on the
elements $\mbz_0$, $\mbz_1$, $\mbz_2$ in Table~\ref{tab:uniGzjSO7}
from L\"ubeck's data. In the table zeros are replaced by dots.

\begin{table}[h]
 \centering
 \[ 
 \begin{array}{|c|ccc|ccc|}
 \hline 
 & \multicolumn{3}{|c|}{q \,\, \text{odd}} & \multicolumn{3}{|c|}{q
     \,\, \text{even}}\\
 \Lambda & \mbz_0 & \mbz_1 & \mbz_2 & \mbz_0 & \mbz_1 & \mbz_2 \\\hline
  {[2,1,1]} & \frac12 q (2 q^2+q+1) & \frac12 q \phi_2^2 &
   \frac12 q \phi_4& \frac12 q (2q^2+q+1) & 
   \frac12 q \phi_2 \phi_4 & \frac12 q \phi_3 \\
  {[-,3,1]} & \frac12 q (2 q^2-q+1)  & \frac12 q \phi_1^2 &
   \frac12 q \phi_4 & \frac12 q (2q^2-q+1) & 
   -\frac12 q \phi_1 \phi_4 & \frac12 q \phi_6\\
  {[2\,1,-,1]} & \frac12 q \phi_2 & \frac12 q \phi_2^2 &
   \frac12 q \phi_4 & \frac12 q \phi_2 & \frac12 q \phi_2 \phi_4 & 
   \frac12 q \phi_3\\
  {[1,-,3]} & -\frac12 q \phi_1 & \frac12 q \phi_1^2 &
   \frac12 q \phi_4 & -\frac12 q \phi_1 & 
   -\frac12 q \phi_1 \phi_4 & \frac12 q \phi_6\\
  {[1,2,1]} & q^2 \phi_4 & q^2 & q^2&
   q^2 \phi_4 & q^2 & q^2\\
  {[1^2,1,1]} & q^3 & 2 q^3 & .&
   q^3 & q^3 \phi_4 & q^3\\
  {[1,1^2,1]} & \frac12 q^4 \phi_2 & q^4 & .&
   \frac12 q^4 \phi_2 & \frac12 q^4 \phi_4 & \frac12 q^4\\
  {[-,2\,1,1]} & \frac12 q^4 \phi_2 & . & q^4&
   \frac12 q^4 \phi_2 & -\frac12 q^4 \phi_1 \phi_2 & \frac12 q^4\\
  {[1^3,-,1]} & -\frac12 q^4 \phi_1 & q^4 & .&
   -\frac12 q^4 \phi_1 & \frac12 q^4 \phi_4 & \frac12 q^4\\
  {[-,1,3]} & -\frac12 q^4 \phi_1 & . & q^4&
   -\frac12 q^4 \phi_1 & -\frac12 q^4 \phi_1 \phi_2 & \frac12 q^4\\
  \hline
 \end{array}
\]
\caption{Values $\chi_\Lambda(\mbz_j)$ of unipotent characters of
  $\SO_7(q)$.}
\label{tab:uniGzjSO7}
\end{table}

Additionally, we have the values $\chi_{[3,-,1]}(\mbz_j) = 1$ for all
$j$ for the trivial character and $\chi_{[-,1^3,1]}(\mbz_j) = 0$ for
all $j$ for the Steinberg character. 

\begin{theorem}\label{thm:unipotresSO7}
For odd $q$ and even $q$ the Type~$1$ and Type~$0$ components of the
restrictions of the unipotent characters of $\SO_7(q)$ to $P_7$ are
given in Table~\ref{tab:unipotresSO7}. 
\end{theorem}
\begin{table}[h]
 \[
 \begin{array}{|c||c|c|c|c|c|}
  \hline
   & \multicolumn{5}{|c|}{\chi_\Lambda\downarrow^G_P}\\
   & \text{Type} \,\, 1 & \multicolumn{4}{|c|}{\text{Type} \,\, 0 }\\
  \Lambda && 1 & 0 & + & -\\
  \hline
  {[3,-,1]} & [2,-,1] & & & & \\
  {[2,1,1]} & [2,-,1]+[1,1,1] & [1,-,1] & & & \\
  {[-,3,1]} & [-,2,1] & & & & \\
  {[2\,1,-,1]} & [2,-,1]+[1^2,-,1] & [1,-,1] & & & \\
  {[1,-,3]} & [-,-,3] & & & & \\
  {[1,2,1]} & [1,1,1]+[-,2,1] & [-,1,1] & & & \\
  {[1^2,1,1]} & [1^2,-,1]+[1,1,1] & [1,-,1]+[-,1,1] & 1_{P_3} & & \\
  {[1,1^2,1]} & [-,1^2,1]+[1,1,1] & [1,-,1]+[-,1,1] & 1_{P_3} & 1_{L_5^+} & \\
  {[-,2\,1,1]} & [-,1^2,1]+[-,2,1] & [-,1,1] & & \nu_1^+ & \\
  {[1^3,-,1]} & [1^2,-,1] & [1,-,1] & 1_{P_3} & & 1_{L_5^-} \\
  {[-,1,3]} & [-,-,3] & & & & \nu_1^- \\
  {[-,1^3,1]} & [-,1^2,1] & [-,1,1] & \hspace{-0.1cm} 1_{P_3}+\mu \hspace{-0.1cm} & 
    \hspace{-0.1cm} 1_{L_5^+}+\nu_1^+(+\nu_3^+)+\Xi^+ \hspace{-0.1cm}
    & \hspace{-0.1cm} (\nu_3^-+)\Xi^- \hspace{-0.1cm} \\\hline 
 \end{array}
 \]
 \caption{Type~$1$ and Type~$0$ components of the restrictions of the
   unipotent characters of $\SO_7(q)$ to $P_7$.} 
\label{tab:unipotresSO7}
\end{table}

\begin{remark} \label{rmk:rmkresSO7}
The first column of Table~\ref{tab:unipotresSO7} contains the symbols
parameterizing the unipotent characters $\chi_\Lambda$ of $G = \SO_7(q)$.
The second column lists symbols parameterizing characters $\sigma$
of $L$ such that ${}^1\psi_\sigma$ is the Type~$1$ component of
$\chi_\Lambda\smash\downarrow^G_P$. The entries in columns 3-6 are
characters $\vartheta^1$, $\vartheta^0$, $\vartheta^+$, $\vartheta^-$,
respectively, such that ${}^0\psi_{({}^1\psi_{\vartheta^1})} + 
{}^0\psi_{({}^0\psi_{\vartheta^0})} +
{}^0\psi_{({}^+\psi_{\vartheta^+})} +
{}^0\psi_{({}^-\psi_{\vartheta^-})}$ is the Type~$0$ component of
$\chi_\Lambda\smash\downarrow^G_P$. 

The characters $\mu$, $\nu_j^\pm$ and $\Xi^\pm$ are defined in
Remark~\ref{rmk:charsubSO5}. The characters in brackets in the last
row of Table~\ref{tab:unipotresSO7} only exist for odd $q$. More
precisely: The Type~$0$ component of
$\chi_{[-,1^3,1]}\smash\downarrow^G_P$ is ${}^0\psi_\Gamma$ where
\[
\Gamma := {}^1\psi_{\chi_{[-,1,1]}} 
+ {}^0\psi_{1_{P_3}} + {}^0\psi_{\mu}
+ {}^+\psi_{1_{L_5^+}} + {}^+\psi_{\nu_1^+} 
+ {}^+\psi_{\nu_3^+} + {}^+\psi_{\Xi^+}
+ {}^-\psi_{\nu_3^-} + {}^-\psi_{\Xi^-}
\]
for odd $q$ and
\[
\Gamma := {}^1\psi_{\chi_{[-,1,1]}} 
+ {}^0\psi_{1_{P_3}} + {}^0\psi_{\mu}
+ {}^+\psi_{1_{L_5^+}} + {}^+\psi_{\nu_1^+} 
+ {}^+\psi_{\Xi^+} + {}^-\psi_{\Xi^-}
\]
for even $q$.
\end{remark}

\begin{proof} (of Theorem~\ref{thm:unipotresSO7})
Let $\chi_\Lambda$ be a unipotent character of $G$. As in
Remark~\ref{rmk:valdegs} we write the restriction of $\chi_\Lambda$ to
$P$ as $\chi_\Lambda\smash\downarrow^G_P = {}^1\chi + {}^0\chi +
{}^+\chi + {}^-\chi$ where 
${}^\varepsilon\chi = {}^\varepsilon\psi_{\vartheta^\varepsilon}$ is
the Type~$\varepsilon$ component of $\chi_\Lambda\smash\downarrow^G_P$. 
The degrees $\vartheta^\varepsilon(1)$ can be computed from
Table~\ref{tab:uniGzjSO7} via~\eqref{eq:vals_mat}. The result is given in
Table~\ref{tab:resSO7degs}; it turns out that it does not depend on
whether $q$ is odd or even.

\begin{table}
 \[
 \begin{array}{|c|cccc|}
 \hline \Lambda & \vartheta^1(1) & \vartheta^0(1) & \vartheta^+(1) & \vartheta^-(1)\\\hline
  {[3,-,1]} & 1 & . & . & .\\
  {[2,1,1]} &  \frac12 (q+2)(q^2+1) & 1 & 1 & .\\
  {[-,3,1]} & \frac12 q(q^2+1) & . & 1 & .\\
  {[2\,1,-,1]} & \frac12 (q+1) (q^2-q+2) & 1 & . & 1\\
  {[1,-,3]} & \frac12 q (q-1)^2 & . & . & 1\\
  {[1,2,1]} & q (q^2+q+1) & q & 2 q & .\\
  {[1^2,1,1]} & q (q^2+q+1) & q (q+1) & q^2 & q^2\\
  {[1,1^2,1]} & \frac12 q (2q+1)(q^2+1) & \frac12 q (q+1)^2 & \frac12 q (q+1)^2 & \frac12 q (q^2+1)\\
  {[-,2\,1,1]} & \frac12 q (q^2+2q^3+1) & \frac12 q (q^2+1) & \frac12 q (q+1)^2 & \frac12 q (q^2+1)\\
  {[1^3,-,1]} & \frac12 q (q^2+1) & \frac12 q (q^2+1) & \frac12 q (q-1)^2 & \frac12 q (q^2+1)\\
  {[-,1,3]} & \frac12 q(q-1)^2 & \frac12 q (q-1)^2 & \frac12 q (q-1)^2 & \frac12 q (q^2+1)\\
  {[-,1^3,1]} & q^4 & q^4 & q^4 & q^4\\\hline
 \end{array}
 \]
 \caption{Component degrees of the restrictions of the unipotent
   irreducible characters of $\SO_7(q)$ to $P_7$ for both odd $q$ and 
   even $q$.} 
\label{tab:resSO7degs}
\end{table}

The first row of Table~\ref{tab:unipotresSO7} is trivial and rows 3 
and 5 follow from Table~\ref{tab:resSO7degs}. Also, the
first column of Table~\ref{tab:unipotresSO7} can be determined from 
Table~\ref{tab:resSO7degs} and the combinatorics of bipartitions which
encode Harish-Chandra induction and restriction. In particular, we
obtain 
\[ 
\R{G}{L}(1_L) = \R{G}{L}(\chi_{[2,-,1]}) = \chi_{[3,-,1]} + \chi_{[2,1,1]} + \chi_{[21,-,1]}.
\]
By Proposition~\ref{prop:resSt_SO} we have 
$1_L\smash\uparrow_L^P = {^1}\psi_{1_L} + {^0}\psi_{1_{P_5}} +
{^+}\psi_{1_{L^+}} + {^-}\psi_{1_{L^-}}$. Theorem~\ref{thm:indnu}~(b)
and the paragraph preceding it now yield for $\sigma = 1_L$,
$\Sigma = \chi_{[2,-,1]}+\chi_{[1^2,-,1]}+\chi_{[1,1,1]}$ the equation 
\begin{eqnarray*}
\R{G}{L}(1_L)\downa{G}{P} & = & 1_P + 1_{RQ_K}\smash\uparrow_{RQ_K}^P + 1_{L}\smash\uparrow_L^P\\
  & = & 1_P + {^0}\psi_{1_{P_5}} + {^1}\psi_\Sigma + 
  {^1}\psi_{1_L} + {^0}\psi_{1_{P_5}} +{^+}\psi_{1_{L^+}} +
  {^-}\psi_{1_{L^-}}.
\end{eqnarray*}
Now the entries in rows 2 and 4 follow from
Table~\ref{tab:resSO7degs}. We set $\eta := \chi_{[-,-,3]} \in \Irr(L)$. 
We have 
\begin{equation} \label{eq:etaHC}
\R{G}{L}(\eta) = \R{G}{L}(\chi_{[-,-,3]}) = \chi_{[1,-,3]}+\chi_{[-,1,3]}.
\end{equation}
Since conjugation with $t$ permutes the unipotent characters of $L$
and since $\eta$ is the unique non-trivial unipotent character of $L$
of degree $\frac12 q (q-1)^2$ we have ${}^t\eta = \eta$. Hence, we get
from~\eqref{eq:LGP} that
\[
\R{G}{L}(\eta)\smash\downarrow^G_P = {}^1\psi_\eta + 
{}^s\eta\smash\downarrow^{{}^sP}_{RQ_K}\smash\uparrow_{RQ_K}^P +
\eta\smash\uparrow_L^P.
\]
By Proposition~\ref{prop:resSt_SO} and
Theorem~\ref{thm:unipotresSO5} the character ${}^0\psi_{({}^-\psi_{\nu_1^-})}$ 
is a constituent of~$\eta\smash\uparrow_L^P$. From~\eqref{eq:etaHC} we
get $\R{G}{L}(\eta)\smash\downarrow^G_P = 
\chi_{[1,-,3]}\smash\downarrow^G_P+\chi_{[-,1,3]}\smash\downarrow^G_P$,
and since $\chi_{[1,-,3]}\smash\downarrow^G_P$ does not have any
Type~$0$ component, we see that ${}^0\psi_{({}^-\psi_{\nu_1^-})}$ is a 
constituent of $\chi_{[-,1,3]}\smash\downarrow^G_P$. From the degree
${}^-\psi_{\nu_1^-}(1) = \frac12 q (q-1)^2$ and Table~\ref{tab:resSO7degs}
we can conclude that ${}^0\psi_{({}^-\psi_{\nu_1^-})}$ is the Type~$0$
component of $\chi_{[-,1,3]}\smash\downarrow^G_P$ proving the entries
in row 11 of Table~\ref{tab:unipotresSO7}. We have 
\begin{equation} \label{eq:etaHC-21}
\R{G}{L}(\chi_{[-,2,1]}) = \chi_{[-,3,1]}+\chi_{[1,2,1]}+\chi_{[-,21,1]}.
\end{equation}
Conjugation with $t$ permutes the unipotent characters of $L$. The
characters $\chi_{[-,2,1]}$ and $\chi_{[1^2,-,1]}$ are the only unipotent
characters of $L$ of degree $\frac12 q(q^2+1)$. Since the conjugacy
class of $L$ containing $\mbz_0$ is fixed by conjugation with $t$ and
since the values of $\chi_{[-,2,1]}$ and $\chi_{[1^2,-,1]}$ on this class
differ we have ${}^t\chi_{[-,2,1]}=\chi_{[-,2,1]}$ and 
${}^t\chi_{[1^2,-,1]}=\chi_{[1^2,-,1]}$. Hence, we get from~\eqref{eq:LGP} 
that
\[
\R{G}{L}(\chi_{[-,2,1]})\smash\downarrow^G_P = {}^1\psi_{\chi_{[-,2,1]}} + 
{}^s\chi_{[-,2,1]}\smash\downarrow^{{}^sP}_{RQ_K}\smash\uparrow_{RQ_K}^P +
\chi_{[-,2,1]}\smash\uparrow_L^P.
\]
By Theorem~\ref{thm:unipotresSO5} and Theorem~\ref{thm:indnu}~(b) the
character ${}^0\psi_{{}^1\psi_{\chi_{[-,1,1]}}}$ is a constituent of 
${}^s\chi_{[-,2,1]}\smash\downarrow^{{}^sP}_{RQ_K}\smash\uparrow_{RQ_K}^P$.
Furthermore, it follows from Theorem~\ref{thm:unipotresSO5} and
Proposition~\ref{prop:resSt_SO} that 
${}^0\psi_{{}^1\psi_{\chi_{[-,1,1]}}} + {}^0\psi_{({}^+\psi_{\nu_1^+})}$ 
is a subcharacter of $\chi_{[-,2,1]}\smash\uparrow_L^P$. Thus,
\eqref{eq:etaHC-21} implies that 
$2 \cdot {}^0\psi_{{}^1\psi_{\chi_{[-,1,1]}}} + {}^0\psi_{({}^+\psi_{\nu_1^+})}$ 
is a subcharacter of $\chi_{[-,3,1]}\smash\downarrow^G_P + 
\chi_{[1,2,1]}\smash\downarrow^G_P + \chi_{[-,21,1]}\smash\downarrow^G_P$.
Since $\chi_{[-,3,1]}\smash\downarrow^G_P$ does not have any
Type~$0$ component, we see that 
$2 \cdot {}^0\psi_{{}^1\psi_{\chi_{[-,1,1]}}} + {}^0\psi_{({}^+\psi_{\nu_1^+})}$ 
is a subcharacter of the restriction $\chi_{[1,2,1]}\smash\downarrow^G_P +
\chi_{[-,21,1]}\smash\downarrow^G_P$. Because
${}^1\psi_{\chi_{[-,1,1]}}(1) = q$ and ${}^+\psi_{\nu_1^+}(1) =
\frac12 q (q^2-1)$ we get from Table~\ref{tab:resSO7degs} that 
${}^0\psi_{{}^1\psi_{\chi_{[-,1,1]}}}$ is the Type~$0$ component of 
$\chi_{[1,2,1]}\smash\downarrow^G_P$ and
${}^0\psi_{{}^1\psi_{\chi_{[-,1,1]}}} + {}^0\psi_{({}^+\psi_{\nu_1^+})}$  
is the Type~$0$ component of $\chi_{[-,21,1]}\smash\downarrow^G_P$.
This proves the entries in rows 6 and 9 of Table~\ref{tab:unipotresSO7}. 

From Corollary~\ref{cor:resSt_SO} and Theorem~\ref{thm:unipotresSO5}
we get that the Type~$0$ component of the restriction 
$\chi_{[-,1^3,1]}\smash\downarrow^G_P = \St_G\smash\downarrow^G_P$ is
\begin{gather*}
{}^0\psi_{({}^1\psi_{\chi_{[-,1,1]}})} 
+ {}^0\psi_{({}^0\psi_{1_{P_3}})} 
+ {}^0\psi_{({}^0\psi_\mu)} 
+ {}^0\psi_{({}^+\psi_{1_{L_5^+}})}  
+ {}^0\psi_{({}^+\psi_{1_{\nu_1^+}})}  
\left(+ {}^0\psi_{({}^+\psi_{1_{\nu_3^+}})}\right)\\
+ {}^0\psi_{({}^+\psi_{\Xi^+})}  
\left(+ {}^0\psi_{({}^-\psi_{1_{\nu_3^-}})}\right)
+ {}^0\psi_{({}^-\psi_{\Xi^-})}
\end{gather*}
where the summands in brackets only occur for odd $q$. This proves the
entries in row 12 of Table~\ref{tab:unipotresSO7}. We have 
\begin{equation} \label{eq:etaHC-1h2}
\R{G}{L}(\chi_{[-,1^2,1]}) = \chi_{[1,1^2,1]}+\chi_{[-,21,1]}+\chi_{[-,1^3,1]}.
\end{equation}
Since conjugation with $t$ permutes the unipotent characters of $L$
we have ${}^t\chi_{[-,1^2,1]}={}^t\St_L = \chi_{[-,1^2,1]}$. Hence, we get
from~\eqref{eq:LGP} and \cite[Proposition~6.3.3]{Carter} that
\begin{eqnarray*}
\R{G}{L}(\chi_{[-,1^2,1]})\smash\downarrow^G_P & = & {}^1\psi_{\chi_{[-,1^2,1]}} + 
{}^s\chi_{[-,1^2,1]}\smash\downarrow^{{}^sP}_{RQ_K}\smash\uparrow_{RQ_K}^P +
\chi_{[-,1^2,1]}\smash\uparrow_L^P\\
& = & {}^1\psi_{\chi_{[-,1^2,1]}} + 
{}^s\chi_{[-,1^2,1]}\smash\downarrow^{{}^sP}_{RQ_K}\smash\uparrow_{RQ_K}^P +
\chi_{[-,1^3,1]}\smash\downarrow^G_P.
\end{eqnarray*}
Together with~\eqref{eq:etaHC-1h2} this implies that the Type~$0$
component of
$\chi_{[1,1^2,1]}\smash\downarrow^G_P+\chi_{[-,21,1]}\smash\downarrow^G_P$ 
is the Type~$0$ component of 
${}^s\chi_{[-,1^2,1]}\smash\downarrow^{{}^sP}_{RQ_K}\smash\uparrow_{RQ_K}^P$.
We obtain from Theorems~\ref{thm:unipotresSO5} and \ref{thm:indnu}~(b), (c)
that ${}^0\psi_{({}^1\psi_{\chi_{[-,1,1]}})} + {}^0\psi_\Sigma$ is a subcharacter of 
${}^s\chi_{[-,1^2,1]}\smash\downarrow^{{}^sP}_{RQ_K}\smash\uparrow_{RQ_K}^P$,
where $\Sigma := 1_{{}^rP_5 \cap P_5}\smash\uparrow_{{}^rP_5 \cap P_5}^{P_5}$.
The permutation character $\Sigma$ can be computed via Theorem~\ref{thm:indnu}
(b), and we get
$\Sigma = {}^0\psi_{1_{P_3}} + {}^1\psi_{1_{L_5}} + {}^1\psi_{\St_{L_5}}$
so that
\[
{}^0\psi_{({}^1\psi_{\chi_{[-,1,1]}})} + {}^0\psi_{({}^0\psi_{1_{P_3}})} +
{}^0\psi_{({}^1\psi_{\chi_{[1,-,1]}})} + {}^0\psi_{({}^1\psi_{\chi_{[-,1,1]}})} 
\]
is a subcharacter of the Type~$0$ component of
$\chi_{[1,1^2,1]}\smash\downarrow^G_P+\chi_{[-,21,1]}\smash\downarrow^G_P$.
Because we already know the Type~$0$ component of
$\chi_{[-,21,1]}\smash\downarrow^G_P$ we see that 
\[
{}^0\psi_{({}^0\psi_{1_{P_3}})} +
{}^0\psi_{({}^1\psi_{\chi_{[1,-,1]}})} +
{}^0\psi_{({}^1\psi_{\chi_{[-,1,1]}})}  
\]
is a subcharacter of the Type~$0$ component of $\chi_{[1,1^2,1]}\smash\downarrow^G_P$.
Next, we consider the Type~$0$ component of
$\R{G}{L}(\chi_{[1,1,1]})\smash\downarrow^G_P$. We have
\begin{equation} \label{eq:etaHC-111}
\R{G}{L}(\chi_{[1,1,1]}) = \chi_{[2,1,1]}+\chi_{[1,2,1]}+\chi_{[1^2,1,1]}+\chi_{[1,1^2,1]}.
\end{equation}
As before, we see from the degrees that conjugation with $t$ fixes
$\chi_{[1,1,1]}$. Hence we get from from~\eqref{eq:LGP} that
\[
\R{G}{L}(\chi_{[1,1,1]})\smash\downarrow^G_P = {}^1\psi_{\chi_{[1,1,1]}} + 
{}^s\chi_{[1,1,1]}\smash\downarrow^{{}^sP}_{RQ_K}\smash\uparrow_{RQ_K}^P +
\chi_{[1,1,1]}\smash\uparrow_L^P.
\]
Via Theorems~\ref{thm:unipotresSO5} and \ref{thm:indnu}~(b),(c) we
can compute the Type~$0$ component of 
${}^s\chi_{[1,1,1]}\smash\downarrow^{{}^sP}_{RQ_K}\smash\uparrow_{RQ_K}^P$
and get:
\[
{}^0\psi_{({}^1\psi_{\chi_{[1,-,1]}})} + {}^0\psi_{({}^1\psi_{\chi_{[-,1,1]}})} + {}^0\psi_\Sigma = 
2 \cdot {}^0\psi_{({}^1\psi_{\chi_{[1,-,1]}})} + 2 \cdot {}^0\psi_{({}^1\psi_{\chi_{[-,1,1]}})} +
{}^0\psi_{({}^0\psi_{1_{P_3}})}.
\]
The Type~$0$ component of $\chi_{[1,1,1]}\smash\uparrow_L^P$ can be
determined with the help of Theorem~\ref{thm:unipotresSO5} and
Proposition~\ref{prop:resSt_SO} and we get:
\[
{}^0\psi_{({}^1\psi_{\chi_{[1,-,1]}})} + {}^0\psi_{({}^1\psi_{\chi_{[-,1,1]}})} +
{}^0\psi_{({}^0\psi_{1_{P_3}})} + {}^0\psi_{({}^+\psi_{1_{L_5^+}})}.
\]
Together with \eqref{eq:etaHC-111} it follows that the Type~$0$
component of
$\chi_{[2,1,1]}\smash\downarrow^G_P +
\chi_{[1,2,1]}\smash\downarrow^G_P +
\chi_{[1^2,1,1]}\smash\downarrow^G_P + \chi_{[1,1^2,1]}\smash\downarrow^G_P$ is:
$3 \cdot {}^0\psi_{({}^1\psi_{\chi_{[1,-,1]}})} + 3 \cdot {}^0\psi_{({}^1\psi_{\chi_{[-,1,1]}})} +
2 \cdot {}^0\psi_{({}^0\psi_{1_{P_3}})} + {}^0\psi_{({}^+\psi_{1_{L_5^+}})}$.
Since we already know the Type~$0$ component of $\chi_{[2,1,1]}\smash\downarrow^G_P +
\chi_{[1,2,1]}\smash\downarrow^G_P$ we see that the Type~$0$ component
of $\chi_{[1^2,1,1]}\smash\downarrow^G_P + \chi_{[1,1^2,1]}\smash\downarrow^G_P$ is:
\begin{equation} \label{eq:1h21+11h2type0}
2 \cdot {}^0\psi_{({}^1\psi_{\chi_{[1,-,1]}})} + 2 \cdot {}^0\psi_{({}^1\psi_{\chi_{[-,1,1]}})} +
2 \cdot {}^0\psi_{({}^0\psi_{1_{P_3}})} + {}^0\psi_{({}^+\psi_{1_{L_5^+}})}.
\end{equation}
Next, we consider the Type~$0$ component of
$\R{G}{L}(\chi_{[1^2,-,1]})\smash\downarrow^G_P$. We have
\begin{equation} \label{eq:etaHC-1h2m}
\R{G}{L}(\chi_{[1^2,-,1]}) = \chi_{[21,-,1]}+\chi_{[1^2,1,1]}+\chi_{[1^3,-,1]}.
\end{equation}
We have already seen above that conjugation with $t$ fixes $\chi_{[1^2,-,1]}$.
Hence we get from from~\eqref{eq:LGP} that
\[
\R{G}{L}(\chi_{[1^2,-,1]})\smash\downarrow^G_P = {}^1\psi_{\chi_{[1^2,-,1]}} + 
{}^s\chi_{[1^2,-,1]}\smash\downarrow^{{}^sP}_{RQ_K}\smash\uparrow_{RQ_K}^P +
\chi_{[1^2,-,1]}\smash\uparrow_L^P.
\]
Via Theorems~\ref{thm:unipotresSO5} and \ref{thm:indnu}~(b),(c) we
can compute the Type~$0$ component of 
${}^s\chi_{[1^2,-,1]}\smash\downarrow^{{}^sP}_{RQ_K}\smash\uparrow_{RQ_K}^P$
and get:
\[
{}^0\psi_{({}^1\psi_{\chi_{[1,-,1]}})} + {}^0\psi_\Sigma = 
{}^0\psi_{({}^1\psi_{\chi_{[1,-,1]}})} + {}^0\psi_{({}^0\psi_{1_{P_3}})} +
{}^0\psi_{({}^1\psi_{\chi_{[1,-,1]}})} + {}^0\psi_{({}^1\psi_{\chi_{[-,1,1]}})}.
\]
The Type~$0$ component of $\chi_{[1^2,-,1]}\smash\uparrow_L^P$ can be
determined via Theorem~\ref{thm:unipotresSO5} and
Proposition~\ref{prop:resSt_SO}. One gets:
${}^0\psi_{({}^1\psi_{\chi_{[1,-,1]}})} + {}^0\psi_{({}^0\psi_{1_{P_3}})} 
+ {}^0\psi_{({}^-\psi_{1_{L_5^-}})}$. 
Together with \eqref{eq:etaHC-1h2m} it follows that the Type~$0$
component of $\chi_{[21,-,1]}\smash\downarrow^G_P +
\chi_{[1^2,1,1]}\smash\downarrow^G_P + \chi_{[1^3,-,1]}\smash\downarrow^G_P$ is:
\[
3 \cdot {}^0\psi_{({}^1\psi_{\chi_{[1,-,1]}})} + 
{}^0\psi_{({}^1\psi_{\chi_{[-,1,1]}})} + 2 \cdot
{}^0\psi_{({}^0\psi_{1_{P_3}})} + {}^0\psi_{({}^-\psi_{1_{L_5^-}})}. 
\]
Since we already know the Type~$0$ component of
$\chi_{[21,-,1]}\smash\downarrow^G_P$ we see that the Type~$0$
component of $\chi_{[1^2,1,1]}\smash\downarrow^G_P + \chi_{[1^3,-,1]}\smash\downarrow^G_P$ is:
\begin{equation} \label{eq:1h21+1h3mtype0}
2 \cdot {}^0\psi_{({}^1\psi_{\chi_{[1,-,1]}})} + 
{}^0\psi_{({}^1\psi_{\chi_{[-,1,1]}})} + 2 \cdot
{}^0\psi_{({}^0\psi_{1_{P_3}})} + {}^0\psi_{({}^-\psi_{1_{L_5^-}})}. 
\end{equation}
Equation~\eqref{eq:1h21+1h3mtype0} implies that
${}^0\psi_{({}^+\psi_{1_{L_5^+}})}$ is no constituent of
$\chi_{[1^2,1,1]}\smash\downarrow^G_P$ and hence it follows
from~\eqref{eq:1h21+11h2type0} that
${}^0\psi_{({}^+\psi_{1_{L_5^+}})}$ is a constituent of 
$\chi_{[1,1^2,1]}\smash\downarrow^G_P$. Hence
\begin{equation} \label{eq:11h2type0}
{}^0\psi_{({}^0\psi_{1_{P_3}})} +
{}^0\psi_{({}^1\psi_{\chi_{[1,-,1]}})} +
{}^0\psi_{({}^1\psi_{\chi_{[-,1,1]}})} +
{}^0\psi_{({}^+\psi_{1_{L_5^+}})}
\end{equation}
is a subcharacter of the Type~$0$ component of $\chi_{[1,1^2,1]}\smash\downarrow^G_P$.
Comparing degrees with Table~\ref{tab:resSO7degs} we see that
\eqref{eq:11h2type0} is the Type~$0$ component of
$\chi_{[1,1^2,1]}\smash\downarrow^G_P$. Now the Type~$0$ components of
$\chi_{[1^2,1,1]}\smash\downarrow^G_P$ and $\chi_{[1^3,-,1]}\smash\downarrow^G_P$
follow from \eqref{eq:1h21+11h2type0} and \eqref{eq:1h21+1h3mtype0}.
This gives the entries in rows 7,8,10 of Table~\ref{tab:unipotresSO7}
and completes the proof of the theorem.
\end{proof}

\begin{remark} \label{rmk:Types+-} \hspace{-0.2cm}
The proof of Theorem~\ref{thm:unipotresSO7} gives partial
information on the Type~$+$ and Type~$-$ components of the
restrictions of the unipotent characters of $G=\SO_7(q)$ to the
maximal parabolic subgroup~$P=P_7$. In particular, we see that the 
restrictions of $[3,-,1]$, $[21,-,1]$, $[1,-,3]$ to $P$ do not
have any constituents of Type~$+$ and that the Type $+$ component of 
$[2,1,1]\smash\downarrow^G_P$ is ${}^+\psi_{1_{L^+}}$. Furthermore,
the restrictions of $[3,-,1]$, $[2,1,1]$, $[-,3,1]$, $[1,2,1]$ to $P$
do not have any constituents of Type~$-$ and the Type $-$ component of 
$[21,-,1]\smash\downarrow^G_P$ is ${}^-\psi_{1_{L^-}}$. 

The constituents of Type $+/-$ are parameterized by irreducible
characters of the groups $\GO_4^\pm(q)$. At present, 
there is only limited information on the irreducible characters of
these groups. The character tables of these groups and the remaining
Type $+/-$ components will be treated in a forthcoming project. 
\end{remark}

\subsection*{Acknowledgements}

We thank Frank L\"ubeck for fruitful discussions and for providing
the unipotent characters of $\SO_5(q)$ and $\SO_7(q)$. 
We also thank Gerhard Hiss for many insightful discussions.


\newcommand{\etalchar}[1]{$^{#1}$}

\end{document}